\documentclass[journal]{IEEEtran}

\usepackage{cite}

\ifCLASSINFOpdf
 \usepackage[pdftex]{graphicx}
 \usepackage{subfigure}
\else
\fi
\usepackage{amsmath}
\usepackage{amssymb}
\usepackage{amsthm}

\usepackage{color}      
 \usepackage{mathrsfs}

\usepackage{float} 

\newtheorem{lemma}{\textbf{Lemma}}

\newtheorem{theorem}{\textbf{Theorem}}

\newtheorem{assumption}{\textbf{Assumption}}

\newcommand{\prob}[1]{\text{Pr}\big\{#1\big\}}
\newcommand{\expect}[1]{\mathbb{E}\big\{#1\big\}}

\newcommand{\bv}[1]{{\boldsymbol{#1} }}
\newcommand{\script}[1]{{{\cal{#1} }}}

\newtheorem{definition}{\textbf{Definition}}

\newcommand{\huang}[1]{\textcolor{black}{#1}} 
\newcommand{\mtt}[1]{$\mathtt{#1}$}
\newcommand{\bisc}{$\mathtt{BISC}$}
\newcommand{\lbisc}{$\mathtt{LBISC}$}

\newcommand{\mle}{$\mathtt{MLE}$}

 \newenvironment{achievements}{\begin{list}{$\bullet$}{\topsep 0.05pt \itemsep -.1pt}}{\vspace*{1pt}\end{list}}

\allowdisplaybreaks
\begin{document}

\title{System Intelligence: Model, Bounds and Algorithms}

\author{Longbo Huang\\
IIIS@Tsinghua University\\
longbohuang@tsinghua.edu.cn
\thanks{This paper will be presented in part at the $17$th ACM International Symposium on Mobile Ad Hoc Networking and Computing (MobiHoc),  Paderborn, Germany, July 2016.}
\thanks{This work  was supported in part by the National Basic Research Program of China Grant 2011CBA00300, 2011CBA00301, the National Natural Science Foundation of China Grant 61361136003, 61303195, Tsinghua Initiative Research Grant, Microsoft Research Asia Collaborative Research Award, and the China Youth 1000-talent Grant.} 
} 

\maketitle


\begin{abstract} 
We present a general framework for understanding system intelligence, i.e., the level of system smartness perceived by users, and propose a novel metric for measuring intelligence levels of dynamical human-in-the-loop systems, defined to be the maximum average reward obtained by proactively serving user demands, subject to a resource constraint. Our metric captures two important elements of smartness, i.e., being able to know what users want and pre-serve them, and  achieving good resource management while doing so. We provide an explicit characterization of the system intelligence, and show that it is jointly determined by user demand volume (opportunity to impress),   demand correlation (user predictability), and system resource and action costs (flexibility to pre-serve).   

We then propose an online learning-aided control algorithm called Learning-aided Budget-limited Intelligent System Control (\mtt{LBISC}). We show that \lbisc{} achieves an intelligence level that is within $O(N(T)^{-\frac{1}{2}}+\epsilon)$ of the highest level, where $N(T)$ represents the number of data samples collected within a learning period $T$ and is proportional to the user population size in the system,  while guaranteeing an $O(\max( N(T)^{-\frac{1}{2}}/\epsilon, \log(1/\epsilon)^2))$ average resource deficit. Moreover, we show that \lbisc{} possesses an $O(\max( N(T)^{-\frac{1}{2}}/\epsilon$, $ \log(1/\epsilon)^2)+T)$ convergence time, which is much smaller compared to the $\Theta(1/\epsilon)$ time required for non-learning based algorithms. 
The analysis of \lbisc{} rigorously quantifies the impacts of data and user population (captured by $N(T)$), learning (captured by our learning method), and control (captured by \lbisc) on achievable system intelligence, and provides novel insight  and guideline into designing future smart systems.

%

\end{abstract}

\section{Introduction} 
Due to rapid developments in sensing and monitoring, machine learning, and hardware manufacturing, building intelligence into systems has recently received strong attention, and clever technologies and products have been developed to enhance user experience. For instance, recommendation systems \cite{netflix-15}, smart home \cite{smart-home-amazon-echo-16}, artificial intelligence engines  \cite{invest-ai-15}, and user behavior prediction \cite{weber-searching-11}. 
Despite the prevailing success in practice,  there has not been much theoretical understanding about system smartness. In particular, how do we measure the intelligence level of a system? How do we compare two systems and decide which one is smarter? What elements in a system contribute most to the level of smartness? Can the intelligence level of a system be pushed arbitrarily high? 

Motivated by these fundamental questions, in this paper, we propose a general framework for modeling system smartness and propose a novel metric for measuring system intelligence. 
Specifically, we consider a discrete time system where a server serves a set of applications of a user. 
The status of each application changes from time to time, resulting in different demand that needs to be fulfilled. The server observes the demand condition of each application over time and decides at each time whether to \emph{pre-serve} (i.e., serve before the user even place a request) the demand that can come in the next timeslot (which may not be present), or to wait until the next timeslot and serve it then if it indeed arrives. 
Depending on whether demand is served passively or proactively, the server receives different rewards representing user's satisfaction levels, or equivalently, different user  perception of system smartness (a system that can serve us before being asked is often considered smart). On the other hand, due to time-varying service conditions, the service actions incur different costs. 
The objective of the server is to design a control policy that achieves the \emph{system intelligence},  defined to be the maximum achievable reward rate subject to a constraint on average cost expenditure. 

This model models many examples in practice. For example, newsfeed pushing and video prefetching \cite{padmanabhan-96}, \cite{lee-2012}, instant searching \cite{google-instant-13}, and branch prediction in computer architecture \cite{ball-93}, \cite{farooq-2013}. It  captures two key elements of a smart system, i.e., being able to know what users want and pre-execute actions, and performing good resource management while doing so. 
Note that resource management here is critical. Indeed, one can always pre-serve all possible demands to impress users at the expense of very inefficient resource usage, but an intelligent system should do more than that. 

Solving this problem is non-trivial. First of all, rewards generated by actions now depend on their execution timing, i.e., before or after requests. Thus, 
this problem is different from typical network control problems, where outcomes of traffic serving actions  are independent of timing.  
Second, application demands are often correlated over time. Hence, algorithms must be able to handle this issue, and to efficiently explore such correlations. 
Third,  statistical information of the system dynamics are often unknown. Hence,  control algorithms must be able to quickly learn and utilize the information, and must be robust against errors introduced in learning.

 %
%
%

There have been many previous works studying optimal stochastic system control with resource management.  \cite{neelyenergy} designs algorithms for minimizing energy consumption of a stochastic network.  \cite{tan-downlink-09} studies the tradeoff between energy and robustness for downlink systems. \cite{neelysuperfast} and \cite{huangneely_dr_tac} develop algorithms for achieving the optimal utility-delay tradeoff in multihop networks. \cite{hou-delay-10} studies the problem of scheduling delay-constrained flows over wireless systems. However, all these works focus only on causal systems, i.e., service begins only after demand enters the system.  
Recent works \cite{tadrous-proactive}, \cite{xu-mm1}, \cite{zhang-predict-mm1-14}, \cite{xu-necessity-future-info-15} and \cite{huang-predictive-14} consider queueing system control with future traffic demand information. Works \cite{proactive-caching-uncertainties} and \cite{proactive-download-shaping-ton15} also consider similar problems where systems can proactively serve user demand. They obtain interesting results that characterize cost reduction under proactive service and the impact of number of users. 
However, system utilities in the aforementioned works are measured by  average metrics, e.g., throughput or outage probabilities, and actions taken at different times for serving traffic are considered equivalent.  Moreover, they do not investigate the impact of predictability and benefits of learning. 
%

We tackle the problem by first establishing an explicit characterization of the maximum achievable  intelligence level. The characterization provides a fundamental limit of system intelligence and reveals that it is jointly determined by system resource and control costs  (flexibility to pre-serve), user demand volume (opportunity to impress), and demand correlation (user predictability). 
Then, by carefully defining effective rewards and costs that represent action outcomes in pairs of slots, we propose an ideal control algorithm that assumes perfect system statistics, called  Budget-limited Intelligent System Control (\mtt{BISC}). 

We further develop Learning-aided \bisc{} (\lbisc), by incorporating a maximum-likelihood-estimator (\mle) for estimating  statistics, and a dual learning component  (\mtt{DL}) \cite{huang-learning-sig-14} for learning an empirical Lagrange multiplier that can be integrated into \bisc{}, to facilitate algorithm convergence and reduce resource deficit. We show that \lbisc{} achieves a system intelligence that can be pushed arbitrarily close to the highest value while ensuring a deterministic budget deficit bound. %
Furthermore, we investigate the \emph{user-population effect} on algorithm design in system intelligence, and rigorously quantify the degree to which the user population size can impact algorithm performance, i.e.,  algorithm convergence speed can be boosted by a factor that is proportional to the \emph{square-root} of the number of users. 
The analysis of \lbisc{} quantifies how  system  intelligence depends on the amount of system resource, action costs, number of data samples, and the control algorithm. To the best of our knowledge, we are the first to  propose a rigorous metric for quantifying system intelligence, and to jointly analyze the effects of different factors. 

The  contributions of this paper are summarized as follows. 
\begin{itemize}
\item \huang{We propose a mathematical model for investigating system intelligence. Our model captures three important components in smart systems including observation (data), learning and prediction (model training), and algorithm design (control).} 
%
\item \huang{We propose a novel metric for measuring system intelligence, and  explicitly characterize the optimal intelligence level. The characterization shows  that intelligence is jointly determined by system resource and action costs (flexibility to pre-serve), steady-state user demand volume (opportunity to impress),  and the degree of demand correlation (predictability, captured by demand \emph{transition rates}).} 

\item We propose an online learning-aided algorithm, called  Learning-aided Budget-limited Intelligent System Control (\mtt{LBISC}),  for achieving maximum system intelligence. Our algorithm consists of a maximum-likelihood-estimator (\mle) for learning system statistics, a dual-learning  component (\mtt{DL}) for learning a control-critical empirical Lagrange multiplier, and an online queue-based controller based on carefully defined effective action rewards and costs.  

\item We show that  \lbisc{} achieves an intelligence level that is within $O(N(T)^{-1/2}+\epsilon)$ of the maximum, where $T$ is the algorithm learning time, $N(T)$ is the number of data samples collected in learning and is proportional to the user population of  the system, and $\epsilon>0$ is an tunable parameter,  while guaranteeing   that the average  resource deficit is $O(\max( N(T)^{-\frac{1}{2}}/\epsilon, \log(1/\epsilon)^2))$. 
\item We  prove that  \lbisc{} achieves a  convergence time of $O(T+\max( N(T)^{-1/2}/\epsilon, \log(1/\epsilon)^2))$, which can be much smaller than the $\Theta(1/\epsilon)$ time for its non-learning counterpart. 
%
The performance results of \lbisc{} show that a company with more users has significant advantage over those with fewer, in that its algorithm convergence can be boosted by a factor that is proportional to the \emph{square-root} of the user population. 
\vspace{-.05in} 
\end{itemize}

The rest of the paper is organized as follows. In Section \ref{section:example}, we present a few examples of system smartness. We then present our general  model and problem formulation in Section \ref{section:model}, and characterize the optimal system intelligence in Section \ref{section:characterize}. Our algorithms are presented in Section \ref{section:algorithm}. 
Analysis is given in Section \ref{section:analysis}. 
Simulation results are presented in Section \ref{section:simulation}, followed by the conclusion in Section \ref{section:conclusion}.

\section{Examples}\label{section:example} 
In this section, we provide a few examples that will serve both as explanations and motivation for our general model. 

\textbf{Instant searching} \cite{google-instant-13}:  Imagine you are searching on a search engine. When you start typing, the search engine tries to guess whether you will type in a certain keyword and pre-computes  search results that it believes are relevant (predict and pre-service).  
If the server guesses correctly, results can be displayed immediately after typing is done, and  search latency will be significantly reduced, resulting in a great user experience  (high reward). If the prediction is inaccurate, the search engine can still process the query after getting the keyword, with the user being less impressed by the performance (low reward) and resources being wasted computing the wrong results (cost). 

\textbf{Video streaming} \cite{padmanabhan-96}: When a user is watching videos on Youtube or a smart mobile device, the server can predict whether the user wants a particular video clip, and  pre-load the video  to the user device (predict and pre-service). This way, if the prediction is correct,   user experience will be greatly improved and he will enjoy a large satisfaction (high reward). If the prediction is incorrect, the bandwidth and energy spent in pre-loading are wasted (cost), but the server can still stream the content video to on the fly, potentially with a degraded quality-of-service (low reward).





\textbf{Smart home} \cite{learning-auto-heating-ijcai13}: Consider a smart home  environment where a thermostat manages room temperatures in the house. Depending on its prediction about the behavior of  hosts, the thermostat can  pre-heat/pre-cool some of the rooms (predict and pre-service). If the host enters a room where temperature is already adjusted, he receives a high satisfaction (high reward). If the prediction is incorrect, the room temperature can still be adjusted, but may affect user experience (low reward). Moreover, the energy spent is wasted (cost). 

In all these examples, we see that the smart level of a system perceived by users is closely related to whether his demand is served proactively, and whether such predictive service is carried out without too much unnecessary resource expenditure. 
These factors will be made precise in our general given in the next section. 

\vspace{-.1in}
\section{System Model}\label{section:model}
We consider a system where a single server is serving a customer with $M$ applications (Fig. \ref{fig:system}). Here each application can represent,  e.g., a smartphone application, watching a particular video clip, or a certain type of computing task the customer executes regularly. 
We assume that the system operates in slotted time, i.e., $t\in\{0, 1, ...\}$.
\begin{figure}[ht]
\centering
\vspace{-.1in}
\includegraphics[height=1.8in, width=2.4in]{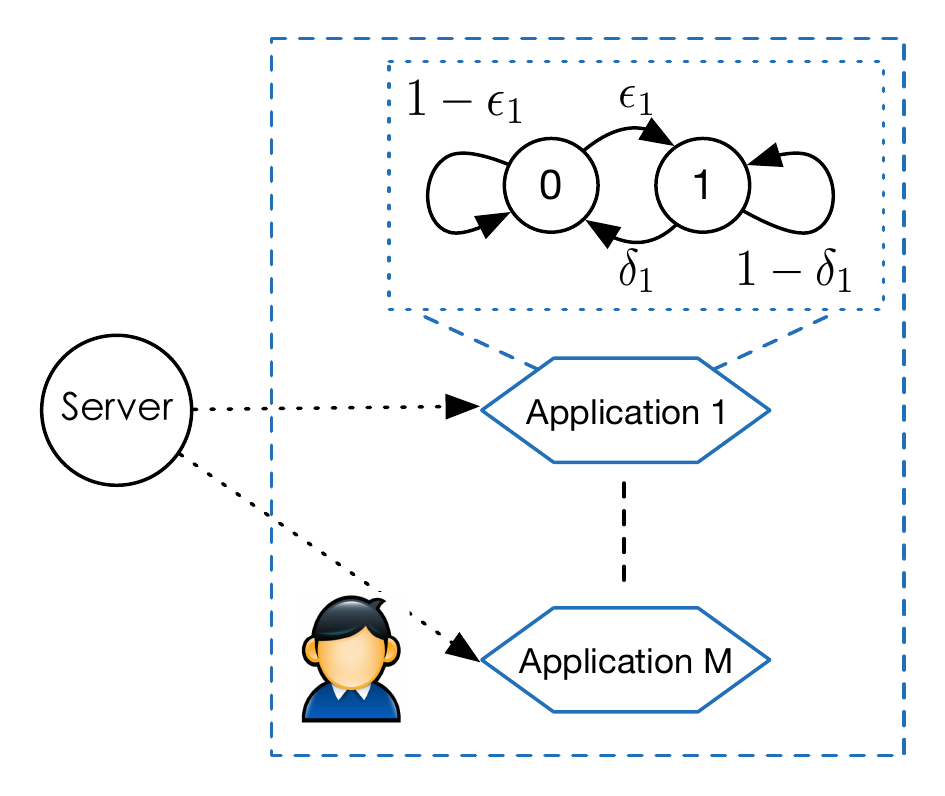}
\vspace{-.1in}
\caption{A multi-application system where a server serves a set of applications of a customer.}\label{fig:system}
\vspace{-.1in}
\end{figure}

\vspace{-.1in}
\subsection{The Demand Model}
We use $\bv{A}(t)=(A_m(t), m=1, ..., M)$ to denote the demand state of the customer at time $t$. We assume that $A_m(t)\in\{0, 1\}$, where $A_m(t)=1$ means there is a unit demand at time $t$ and $A_m(t)=0$ otherwise. 
For instance, if application $m$ represents a video clip watching task, $A_m(t)$ can denote whether the users wants to watch the video clip in the current slot. If so, the server needs to stream the video clip to the customer's device. 

We assume that for each application $m$, $A_m(t)$ evolves according to an independent  two-state Markov chain depicted in Fig. \ref{fig:system}, where the transition probabilities $\epsilon_m$ and $\delta_m$ are as shown in the figure.
Note that such an ON/OFF model has been commonly adopted for modeling network traffic states, e.g., \cite{self-similar-traffic-94} and \cite{dc-traffic-10}.\footnote{\huang{It is possible to adopt a more general multi-state Markov chain for each application.  
Our   results can also likely be extended to model systems where user behavior exhibits periodic patterns. In these scenarios,  Markov models can be built for user's behavior in different time periods and can be learned with data collected in those periods.}}  
We assume that $\epsilon_m$ and $\delta_m$ are unknown to the server, but the actual states can be observed every time slot.\footnote{This is due to the fact that the  states essentially denote whether or not the user has requested service from the server. Thus, by observing the user's response we can see the states.} 

%

\vspace{-.1in}
\subsection{The Service and Cost Model}
In every time slot $t$, the server serves  application $m$'s demand as follows. If $A_m(t)$ has yet been served in slot $t-1$, then it will be served in the current slot. Otherwise the current demand is considered completed. Then, in addition to serving the current demand, the server can also try to \emph{pre-serve} the demand in time $t+1$. 
We denote $\mu_{mc}(t)\in\{0, 1\}$ the action taken to serve the current demand and $\mu_{mp}(t)\in\{0, 1\}$ to denote the action taken to serve the demand in time $t+1$.\footnote{\huang{Our results can be extended to having $\mu_{mp}(t)\in[0, 1]$, in which case partial pre-service is allowed.}} 
It can be seen that:  
\begin{eqnarray}\label{eq:mu-mc-mp}
\mu_{mc}(t) = \max[A_{m}(t)-\mu_{mp}(t-1), 0]. 
\end{eqnarray}
That is,  demand will be fulfilled in the same time slot. Since $\mu_{mc}(t)$ is completely determined by $\mu_{mp}(t-1)$ and $A_m(t)$, we define $\bv{\mu}(t) \triangleq (\mu_{mp}(t), \,\forall\, m)$ for notation simplicity and   view $\bv{\mu}(t)$ as the only control action at   time $t$. 

We assume that each service to application $m$, either proactive or passive, consumes resource of the server. This can be due to energy expenditure or bandwidth consumption.  To capture the fact that the condition under which actions are taken can be time-varying, we denote $S_{m}(t)$ the resource state for application $m$ at time $t$, which affects how much resource is needed for service, e.g., channel condition of a wireless link, or cost spent for getting a particular video clip from an external server. 
We denote $\bv{S}(t)=(S_1(t), ..., S_M(t))$  the overall system resource state, and assume that $\bv{S}(t)\in\script{S}=\{s_1, ..., s_K\}$ with $\pi_k=\prob{\bv{S}(t) = s_k}$ and is i.i.d. every slot (also independent of $\bv{A}(t)$). 
Here we assume that the server can observe the instantaneous state $\bv{S}(t)$ and the $\pi_k$ values are known.\footnote{This assumption is made to allow us to focus on the user demand part. It  is also not restrictive, as $\bv{S}(t)$ is a non-human parameter and can often be learned from observations serving different applications, whereas $\bv{A}(t)$ is more personalized and targeted learning is needed. Nonetheless, our method  also applies to the case when $\{\pi_k, k=1, ..., K\}$ is unknown.} To model the fact that   a given resource state $s_k$ typically constrains the set of feasible actions, we denote $\script{U}_{\bv{S}(t)}$ the set of feasible actions under $\bv{S}(t)$. 
Examples of $\script{U}_{\bv{S}(t)}$ include 
$\script{U}_{s_k} = \{0/1\}^M$ for the unconstrained case, or $\script{U}_{s_k} = \{ \bv{\mu}\in\{0/1\}^M: \sum_{m}\mu_{mp} \leq N_k  \}$ when we are allowed to pre-serve only $N_k$ applications. 
We assume that if $\bv{\mu}\in\script{U}_{s_k}$, then a vector obtained by setting any entry to zero in $\bv{\mu}$ remains in $\script{U}_{s_k}$. 
 
Under the resource state, the instantaneous cost incurred to the server is given by $C(t) =  \sum_mC_m(t)$, where 
\begin{eqnarray}
C_m(t)\triangleq C_m( \mu_{mc}(t), \bv{S}(t) ) + C_m(\mu_{mp}(t), \bv{S}(t) ) \label{eq:cost-def}. 
\end{eqnarray}
%
With (\ref{eq:cost-def}), we assume that the cost in each slot is linear in $\bv{\mu}(t)$, a model that fits situations where costs for serving applications are additive, e.g., amount of bandwidth required for streaming videos. 
We assume that $C_m(\cdot, s_k)=0$ is continuous, $C_m(0, s_k)=0$ and $C_m(1, s_k)\leq C_{\max}$ for some $C_{\max}<\infty$ for all $k$ and $m$. 
We define 
\begin{eqnarray}
C_{\text{av}}\triangleq \limsup_{t\rightarrow\infty}\frac{1}{t}\sum_{\tau=0}^{t-1}\expect{C(\tau)} \label{eq:avg-cost-def}
\end{eqnarray}
as the average cost spent serving the demand. For notation simplicity, we also denote $\overline{C}_{m}\triangleq\sum_k\pi_kC_m(1, s_k)$ as the expected cost for serving one unit demand. 
 
\vspace{-.1in}
\subsection{The Reward Model} 
In each time slot, serving each application demand generates a reward to the server. We use $r_m(t)$ to denote the reward collected in time $t$ from application $m$. It takes the following  values:  
\begin{eqnarray}\label{eq:reward-def}
r_m(t) = \left\{\begin{array}{cc}
0 & A_m(t)=0\\
r_{mc} & A_m(t)=1\,\,\&\,\,\mu_{mp}(t-1)=0 \\
r_{mp} & A_m(t)=1\,\,\&\,\,\mu_{mp}(t-1)=1
\end{array}\right. 
\end{eqnarray}
By varying the values of $r_{mp}$ and $r_{mc}$, we can model different sensitivity levels of the user to pre-service. We assume that  $r_{mp}\geq r_{mc}$ are both known to the server.\footnote{This can be done by monitoring  user feedbacks, e.g., display a short message and ask the user to provide instantaneous feedback. In the case when they are not known a-priori, they can be learned via a similar procedure as in the \lbisc{} algorithm presented later.} This is natural for capturing the fact that a user typically gets more satisfaction if his demand is pre-served. 
We denote $r_{md} = r_{mp}-r_{mc}$ and denote $r_d\triangleq\max_m(r_{mp}-r_{mc})$  the maximum reward difference between pre-service and normal serving. 


To evaluate the performance of the server's control policy, we define the following average reward rate, i.e., 
\begin{eqnarray}
r_{\text{av}} = \liminf_{t\rightarrow\infty}\frac{1}{t}\sum_{\tau=0}^{t-1}\sum_m\expect{r_m(\tau)}.\label{eq:reward-rate} 
\end{eqnarray}
$r_{\text{av}}$ is a natural index of system smartness. A higher value of $r_{\text{av}}$ implies that the server can better predict what the user needs and pre-serves him. As a result, the user perceives a smarter system. 
In the special case when $r_{mp}=1$ and $r_{mc}=0$, the average reward  captures the rate of correct prediction. 
%


\vspace{-.1in}
\subsection{System Objective}
In every time slot, the server accumulates observations about applications, and tries to learn user preferences and to choose proper actions. 
We define $\Gamma$ the set of feasible control algorithms, i.e., algorithms that only choose feasible control actions $\bv{\mu}(t)\in\script{U}_{\bv{S}(t)}$ in every time slot, possibly with help from external information sources regarding application demand statistics. For each policy $\Pi\in\Gamma$, we denote $I(\Pi, \rho)=r_{\text{av}}(\Pi)$ and $C_{\text{av}}(\Pi)$ the resulting algorithm intelligence and average cost rate, respectively.


The objective of the system is to achieve the  {\emph{system intelligence} $I(\rho)$, defined to be the maximum reward rate $r_{\text{av}}$ achievable over all feasible policies, 
subject to a constraint $\rho\in(0, \rho_{\max}]$ on the rate of cost expenditure, i.e.,  
\begin{eqnarray}
\hspace{-.4in} && I(\rho)\,\triangleq\, \max:\quad  I(\Pi, \rho) \label{eq:obj-reward}\\
\hspace{-.4in}&&\qquad\qquad\,\, \text{s.t.}\quad \,\, C_{\text{av}}(\Pi)\leq\rho \label{eq:cond-cost}\\
\hspace{-.4in}  &&\qquad\qquad\qquad\quad \Pi\in\Gamma. \nonumber
\end{eqnarray}
Here  $\rho_{\max}\triangleq \sum_k\pi_k\sum_m\sum C_m(1, s_k)$ is the maximum budget needed to achieve the highest level of intelligence.\footnote{This corresponds to the case of always pre-serving user demand. 
Despite a poor resource utilization, all demands will be pre-served and $I(\rho_{\max})= \sum_m\frac{\epsilon_mr_{mp}}{\epsilon_m+\delta_m}$ where $\frac{\epsilon_m}{\epsilon_m+\delta_m}$ is the steady-state distribution of having $A_m(t)=1$.} 
%

 %

\vspace{-.1in}
\subsection{Discussions of the Model}
\huang{In our model, we have assumed that user demand must be served within the same slot it is placed. This is a suitable model for many task management systems where jobs are time-sensitive, e.g., newsfeed pushing, realtime computation, elevator scheduling, video streaming and searching. 
In these problems, a user's perception about system smartness is often based on whether  jobs are pre-served correctly.}  

Our model captures key ingredients of a general intelligent system including observations, learning and prediction, and control. Indeed, general monitoring and sensing methods can be integrated into the observation part, various learning methods can be incorporated into our learning-prediction step, and control algorithms can be combined with or replace our control scheme presented later. 
%


%

\section{Characterizing intelligence}\label{section:characterize}
In this section, we first obtain a characterization of $I(\rho)$. This result provides interesting insight into system intelligence, and provide a useful criteria  for evaluating smartness of control algorithms. 

To this end,  denote $z(t) = (\bv{A}(t), \bv{S}(t))$ and $\script{Z}=\{z_1, ..., z_H\}$ the state space of $z(t)$, and denote  $\pi_h$ the steady-state distribution of $z_h$. Furthermore,  define for notation simplicity the transition probability $a^{(i_h)}_m$ as: 
\begin{eqnarray}\label{eq:effective-reward-ce}
a^{(i_h)}_m  = \left\{\begin{array}{cc}
1-\delta_m  & \text{when}\quad i_h=A_{m}^{(h)}=1 \\
\epsilon_m  &  \text{when}\quad i_h=A_{m}^{(h)}=0 
\end{array}\right. 
\end{eqnarray}
That is, $a^{(i)}_m$ denotes the probability of having $A_m(t+1)=1$ in the next time slot given the current state. 
Then, our theorem is as follows. 
\begin{theorem}\label{theorem:upper-bdd}
$I(\rho)$ is equal to the optimal value  of the following optimization problem: 
\begin{eqnarray}
\hspace{-.35in}&&\max: \,\,\, \sum_h\pi_h\sum_{j=1}^3\theta^{(h)}_j\sum_m a^{(i_h)}_m [\mu^{(h)}_{mpj}r_{mp}  + (1- \mu^{(h)}_{mpj})r_{mc}] \label{eq:bdd-obj}\\
\hspace{-.35in}&&\quad\text{s.t.} \quad  \sum_h\pi_h\sum_{j=1}^3\theta^{(h)}_j\sum_m [C_m(\mu^{(h)}_{mpj}, z_h)  \\
\hspace{-.35in}&&\qquad\qquad\qquad\qquad\qquad\qquad\qquad   +(1-\mu^{(h)}_{mpj}) a^{(i_h)}_m\overline{C}_m]\leq\rho \nonumber\\
 \hspace{-.35in}&&\qquad \quad\,\,\,  \theta^{(h)}_j\geq0, \,\sum_j\theta^{(h)}_j=1,\,\forall\,z_h, j\nonumber\\
 \hspace{-.35in}&&\qquad \quad\,\,\,  \bv{\mu}_j^{(h)}\in\script{U}_{z_h},\,\forall\,z_h, j. \nonumber
\end{eqnarray}
Here $\theta^{(h)}_{j}$ represents the probability of adopting the pre-service vector $\bv{\mu}_j^{(h)}$ under state $z_h$. 
\end{theorem}
\begin{proof}  
See Appendix A. 
\end{proof}

\huang{Theorem \ref{theorem:upper-bdd} states that no matter what learning and prediction methods are adopted and how the system is controlled, the intelligence level perceived by users will not exceed the value  in (\ref{eq:bdd-obj}). 
This is a powerful result and provides a \emph{fundamental limit} about the system intelligence.  
Theorem \ref{theorem:upper-bdd} also reveals some interesting facts and  rigorously  justifies  various common beliefs about system intelligence. (i) When user demands are more predictable (captured by transition rates $\epsilon_m$ and $\delta_m$, represented by $a^{(i_h)}_m$ in (\ref{eq:bdd-obj})),  the system can achieve a higher  intelligence level.  
(ii) A system with more resources (larger $\rho$) or better cost management (smaller $C_m$ functions) can likely achieve a higher level of perceived smartness. (iii) When there is more demand from users (captured by  distribution $\pi_h$), there are more opportunities for the system to impress the user, and to increase the perceived smartness level.}  
The inclusion of transition rates in the theorem shows that our problem can be  very different from existing network optimization problems, e.g., \cite{ying_wmshortest_infocom09} and \cite{neelyfairness}, where typically steady-state distributions matter most. 

As a concrete example, Fig. \ref{fig:example} shows the $I(\rho)$ values for a single-application  system under (i) $\epsilon=\delta$ and (ii) $\delta=0.6$. 
We see that in the symmetric case, where the steady-state distribution is always $(0.5, 0.5)$, $I(\rho)$ is inverse-proportional to the entropy rate of the demand Markov chain, which  is  consistent with our finding that a higher intelligence level is achievable for more predictable systems (lower entropy). 
For the $\delta=0.6$ case, $I(\rho)$ first increases, then decreases, and then increases again. The reason is as follows. At the beginning, as $\epsilon$ increases, demand increases. Then, when $\epsilon\in[0.4,0.7]$, $I(\rho)$ is reduced by either the increasing randomness (less predictable) or decreased demand (less opportunity). After that, predictability increases and $I(\rho)$ increases again. This  shows that $I(\rho)$ is jointly determined by the steady-state distribution and the transition rates.  
\begin{figure}[ht] 
\centering
\vspace{-.1in}
\includegraphics[height=1.8in, width=3.3in]{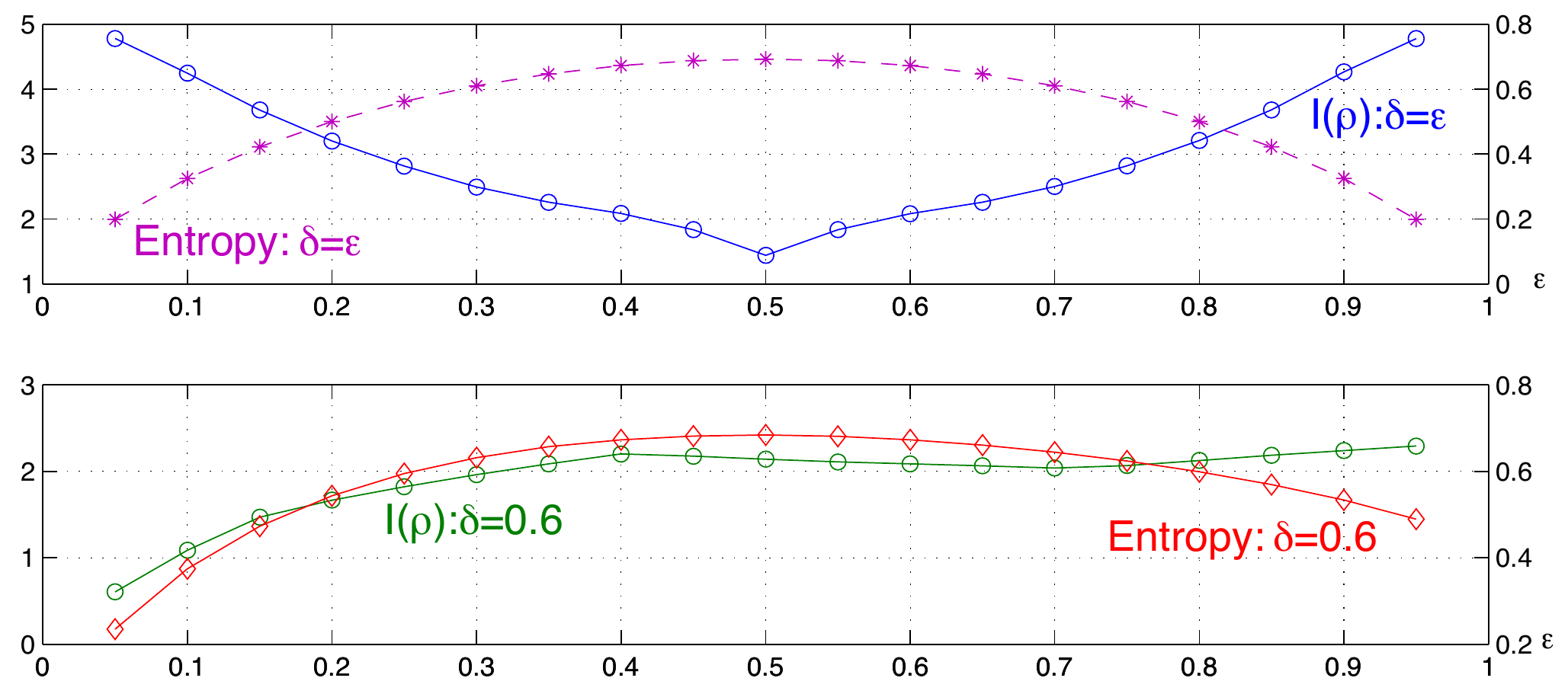}
\vspace{-.05in}
\caption{System intelligence ($M=1$): The left y-axis is for $I(\rho)$  and the right y-axis is for the entropy rate. The x-axis shows the value of $\epsilon$. 
We  see that $I(\rho)$ is consistent with our finding that one can achieve higher intelligence levels for more predictable systems.} 
\label{fig:example}
\vspace{-.1in}
\end{figure}

Note that solving problem (\ref{eq:bdd-obj}) is non-trivial due to the need of system statistics and the potentially complicated structure of $\script{U}_{z_h}$. Thus, in the next section, we propose a learning-based  algorithm for solving the optimal control problem. 
For our algorithm design and analysis, we  define the following  modified dual function for (\ref{eq:bdd-obj}), where $V\geq1$ is a control parameter introduced for later use:\footnote{\huang{Notice that although (\ref{eq:dual}) does not include the $\theta^{(h)}_j$ variables, it can be shown to be equivalent. Moreover, (\ref{eq:dual}) is sufficient for our algorithm design and analysis.}  }
\begin{eqnarray}
\hspace{-.2in}&& g_{\bv{\pi}}(\gamma) \triangleq \sum_h \pi_h\sup_{\bv{\mu}^{(h)}_p}\sum_m\bigg\{ Va^{(i_h)}_m[\mu^{(h)}_{mp}r_{mp}  + (1- \mu^{(h)}_{mp})r_{mc}] \nonumber \\
\hspace{-.2in}&&\qquad\qquad - \gamma  [C_m(\mu^{(h)}_{mp}, z_h)      +(1-\mu^{(h)}_{mp}) a^{(i_h)}_m\overline{C}_m - \rho]\bigg\}. \label{eq:dual}
\end{eqnarray}
Here we use the subscript $\bv{\pi}$ to denote that the dual function is defined with distribution $\bv{\pi}$. 
We also use $\gamma^*$ to denote the minimizer of the dual function. It is shown in \cite{huangneely_qlamarkovian} that: 
\begin{eqnarray}
g_{\bv{\pi}}(\gamma^*)  \geq V\times I(\rho). \label{eq:dual-primal}
\end{eqnarray}
We also define the dual function for each state $z_h$ as follows: 
\begin{eqnarray}
\hspace{-.3in}&& g_{h}(\gamma) \triangleq \sup_{\bv{\mu}^{(h)}_p}\sum_m\bigg\{ Va^{(i_h)}_m[\mu^{(h)}_{mp}r_{mp}  + (1- \mu^{(h)}_{mp})r_{mc}] \label{eq:dual-state}   \\
\hspace{-.3in}&&\qquad\qquad - \gamma  [C_m(\mu^{(h)}_{mp}, z_h)      +(1-\mu^{(h)}_{mp}) a^{(i_h)}_m\overline{C}_m - \rho]\bigg\}. \nonumber
\end{eqnarray} 
It can be  seen that $g_{\bv{\pi}}(\gamma)=\sum_h\pi_hg_{h}(\gamma)$.
%

 


\section{Algorithm Design}\label{section:algorithm}
In this section, we present an online learning-aided control algorithm for achieving the maximum system intelligence. To facilitate understanding, we first present an ideal algorithm that assumes full information of $\epsilon_m$, $\delta_m$, and $\pi_h$. It serves as a building block for our actual algorithm. 
%

\subsection{An Ideal Algorithm}

To do so, we first define the \emph{effective} reward and cost for each application $m$ as functions of $\bv{A}(t)$ and $\bv{\mu}(t)$. Specifically, when $A_m(t)=i$, we have: 
: 
\begin{eqnarray}\label{eq:effective-reward-def}
\tilde{r}^{(i)}_m(\mu_{mp}(t)) = \left\{\begin{array}{cc}
a^{(i)}_mr_{mp}  & \text{if}\quad\mu_{mp}(t)=1 \\ 
a^{(i)}_mr_{mc} &  \text{if}\quad\mu_{mp}(t)=0 
\end{array}\right. 
\end{eqnarray}
Here $a^{(i)}_m$ is defined in (\ref{eq:effective-reward-ce}) as the probability to get into state $A_m(t+1)=1$ conditioning on the current state being $i$. 
To understand the definition, we see that when $A_m(t)=i$, by taking $\mu_{mp}(t)=0$, the server decides not  to pre-serve the potential future demand at $A_m(t+1)$. Hence, if there is demand in slot $t+1$ (happens with probability $a^{(i)}_m$), it  will be served by $\mu_{mc}(t+1)$, resulting in a reward of $r_{mc}$. 
On the other hand, if $\mu_{mp}(t)=1$, with probability $a^{(i)}_m$, the future demand will be pre-served and a reward $r_{mp}$ can be collected. 
It is important to note that the effective reward is defined to be the reward collected in slot $t+1$ as result of actions at time $t$. We denote $\tilde{r}(\bv{\mu}(t)) \triangleq \sum_m\tilde{r}^{(A_m(t))}_m(\mu_{mp}(t))$. 

Similarly, we define the \emph{effective} cost as a function of $\bv{\mu}(t)$ for $A_m(t)=i$:  
\begin{eqnarray}\label{eq:effective-cost-def-1}
\tilde{C}^{(i)}_m(\mu_{mp}(t))  = \left\{\begin{array}{cc}
C_m(1, \bv{S}(t))  &  \text{if}\quad\mu_{mp}(t)=1 \\
a^{(i)}_m\sum_k \pi_kC_m(1, s_k) &  \text{if}\quad\mu_{mp}(t)=0 
\end{array}\right. 
\end{eqnarray}
Note that $\tilde{C}_m(\mu_{mp}(t))$ is the expected cost spent in slots $t$ and $t+1$. As in the effective reward case, we denote $\tilde{C}(\bv{\mu}(t)) \triangleq \sum_m\tilde{C}^{(A_m(t))}_m(\mu_{mp}(t))$. 

With the above definitions, we  introduce a \emph{deficit} queue $d(t)$ that evolves as follows: 
\begin{eqnarray}
d(t+1) = \max[d(t) + \tilde{C}(t) - \rho, 0], \label{eq:deficit-def}
\end{eqnarray}
with $d(0)=0$. 
Then, we define a Lyapunov function $L(t)\triangleq\frac{1}{2}d^2(t)$ and define a single-slot sample-path drift $\Delta(t)\triangleq L(t+1)-L(t)$. 
By squaring both sides of  (\ref{eq:deficit-def}),  using $(\max[x, 0])^2\leq x^2$ for all $x\in\mathbb{R}$, and  $\tilde{C}(\bv{\mu}(t))\leq MC_{\max}$, we obtain the following inequality: 
\begin{eqnarray}
\Delta(t) \leq B - d(t) [\rho-\tilde{C}(t)]. \label{eq:drift-ineq}
\end{eqnarray}
Here $B\triangleq \rho_{\max}^2+M^2C_{\max}^2$. 
Adding to both sides the term $V\sum_m\tilde{r}_m(\mu_{mp}(t))$, where $V\geq1$ is a control parameter, we obtain: 
\begin{eqnarray}\label{eq:drift-ineq-1}
\Delta(t) - V\tilde{r}(\bv{\mu}(t))  \leq B - \bigg(V\tilde{r}(\bv{\mu}(t)) + d(t) [\rho-\tilde{C}(\bv{\mu}(t)) ]\bigg). 
\end{eqnarray} 
Having established (\ref{eq:drift-ineq-1}), we construct the following ideal algorithm by choosing pre-service actions to minimize the right-hand-side of the drift. 

\underline{\textsf{Budget-limited Intelligent System Control} (\mtt{BISC})}: At every time $t$, observe $\bv{A}(t)$, $\bv{S}(t)$ and $d(t)$. Do: 
\begin{achievements}
\item For each $m$, define the cost-differential as follows: 
\begin{eqnarray}
D_m(t) \triangleq C_m(1, \bv{S}(t)) - a^{(i)}_m\overline{C}_m. \label{eq:dm-def}
\end{eqnarray}
Here $i=A_m(t)$. 
Then, solve the following problem to find the optimal pre-serving action $\bv{\mu}(t)$: 
\begin{eqnarray}
\hspace{-.7in}&& \max:\,\,  \sum_{m} \mu_{mp}(t)[Va_m^{(i)} (r_{mp} - r_{mc}) -  d(t) D_m(t)] \label{eq:action-opt}\\
\hspace{-.7in}&&\quad\text{s.t.} \,\, \bv{\mu}(t) \in\script{U}_{\bv{S}(t)}. \nonumber
\end{eqnarray}
\item Update $d(t)$ according to (\ref{eq:deficit-def}). $\Diamond$
\end{achievements} 

A few remarks are in place. (i) The value $a_m^{(i)} (r_{mp} - r_{mc})$ can be viewed as the expected reward loss if we choose $\mu_{mp}(t)=0$ and the value $D_m(t)$ is the expected cost saving for doing so. The parameter $V$ and $d(t)$ thus provide proper weights to the terms for striking a balance between them. 
If the weighted cost saving does not overweight the weighted reward loss, it  is more desirable to pre-serve the demand in the current slot. 
(ii) For applications where $a_m^{(i)} (r_{mp} - r_{mc})$ is smaller, it is less desirable to pre-serve the demand, as the user perception of system intelligence may not be heavily affected. 
(iii) In the special case when $\rho\geq\rho_{\max}$, we see that $d(t)$ will always stay near zero, resulting in $\mu_{mp}(t)=1$  throughout. 
(iv) \mtt{BISC} is easy to implement. Since each $\mu_{mp}(t)$ is either $0$ or $1$, problem (\ref{eq:action-opt}) is indeed finding the maximum-weighted vector from $\script{U}_{\bv{S}(t)}$. In the case when $\script{U}_{\bv{S}(t)}$ only limits the number of non-zero entries, we can easily sort the applications according to the value $Va_m^{(i)} (r_{mp} - r_{mc}) -  d(t) D_m(t)$ and choose the top ones.

\subsection{Learning-aided Algorithm with User Population Effect}\label{subsection:algorithm-general}
In this section, we present an algorithm that  learns $\delta_m$ and $\epsilon_m$   online and performs optimal control simultaneously. 
We also explicitly describe how the system user population size can affect algorithm performance. 

To rigorously quantify this user effect, we first introduce the following \emph{user-population effect} function $N(T)$. 
\begin{definition} 
A system is said to have a user population effect $N(T)$ if within $T$ slots, (i) it collects a sequence of demand samples $\{\bv{A}(0), ..., \bv{A}(N(T)-1)\}$ generated by the Markov process $\bv{A}(t)$, and (ii) $N(T)\geq T$ for all $T$. $\Diamond$
\end{definition}

$N(t)$ captures the number of useful user data samples a system can collect in $T$ time slots, and is a natural indicator about how  user population contributes to learning user  preferences. 
For instance, if there is only one user,   $N(T)=T$. 
On the other hand, if a system has many users,  it can often collect samples from similar users (often determined via machine learning techniques, e.g., clustering) to study a target user's preferences. In this case, one typical example for $N(T)$ can be: 
\begin{eqnarray}
N(T) = f(\#\, \text{of user})\cdot T,  \label{eq:user-population}
\end{eqnarray}
where $f(\#\, \text{of user})$ computes the number of similar users that generate useful samples.

%
%

%
 
Now we  present our algorithm. We begin with  the first component, which is a maximum likelihood estimator (\mtt{MLE}) \cite{walrand-prob-book-14} for estimating the statistics.\footnote{Any alternative estimator that possesses similar features as \mtt{MLE} can also be used here.} 

\underline{\textsf{Maximum Likelihood Estimator} (\mtt{MLE}($T$))}: Fix a learning time $T$ and obtain a sequence of samples $\{\bv{A}(0), ..., \bv{A}(N(T)-1)\}$ in $[0, T-1]$. Output:\footnote{\huang{Here we adopt \mtt{MLE} to demonstrate how learning can be rigorously and efficiently combined with control algorithms to achieve good performance.}} 
\begin{eqnarray}
\hat{\epsilon}_m(T) &=& \frac{\sum_{t=0}^{N(T)-1}1_{\{A_m(t)=0, A_m(t+1)=1\}} }{ \sum_{t=0}^{T-1}1_{\{A_m(t)=0\}} } \label{eq:est-epsilon}\\
\hat{\delta}_m(T) &=& \frac{\sum_{t=0}^{N(T)-1}1_{\{A_m(t)=1, A_m(t+1)=0\}} }{ \sum_{t=0}^{T-1}1_{\{A_m(t)=1\}} }. \label{eq:est-delta}
\end{eqnarray}
That is,   use empirical frequencies to estimate the transition probabilities. $\Diamond$

Note that after estimating $\hat{\bv{\epsilon}}$ and $\hat{\bv{\delta}}$, we also obtain an estimation of $\hat{\bv{\pi}}$. 
We now have the second component, which is a \emph{dual learning} module \cite{huang-learning-sig-14} that learns an empirical Lagrange multiplier based on $\hat{\bv{\epsilon}}$ and $\hat{\bv{\delta}}$, and $\hat{\bv{\pi}}$. 

\underline{\textsf{Dual Learning} (\mtt{DL}($\hat{\bv{\epsilon}}$, $\hat{\bv{\delta}}$, $\hat{\bv{\pi}}$))}: Construct $\hat{g}_{\hat{\bv{\pi}}}(\gamma)$ with $\hat{\bv{\epsilon}}$, $\hat{\bv{\delta}}$, and $\hat{\bv{\pi}}$ according to (\ref{eq:dual}). Solve the following problem and outputs the optimal solution $\gamma^*_T$. 
\begin{eqnarray}
\min: \hat{g}_{\hat{\bv{\pi}}}(\gamma), \,\,\text{s.t.}\,\, \gamma\geq0.\,\,\Diamond \label{eq:dual-learning}
\end{eqnarray}
Here $\hat{g}_{\hat{\bv{\pi}}}(\gamma)$ is the dual function with true statistics being replaced by $\hat{\bv{\epsilon}}$, $\hat{\bv{\delta}}$, and $\hat{\bv{\pi}}$. 
With \mle{}($T$) and \mtt{DL}($\hat{\bv{\epsilon}}$, $\hat{\bv{\delta}}$, $\hat{\bv{\pi}}$), we have our learning-aided \bisc{} algorithm.\footnote{The methodology can be applied to the case when $r_{mp}$ and $r_{mc}$ are also unknown.}

\underline{\textsf{Learning-aided BISC} (\mtt{LBISC}($T$, $\theta$))}: Fix a learning time  $T$ and perform the following:\footnote{The main reason to adopt a finite $T$ is for tractability. In actual implementation, one can continuously refine the estimates for $\hat{\bv{\epsilon}}$, $\hat{\bv{\delta}}$, and $\hat{\bv{\pi}}$ over time.} 
\begin{achievements} 
\item (Estimation) For $t=0, ..., T-1$,  choose $\bv{\mu}_p(t)=\bv{1}$, i.e., always pre-serve user demand.  At time $T$, perform \mtt{MLE}($T$) to obtain $\hat{\bv{\epsilon}}$ and $\hat{\bv{\delta}}$, and $\hat{\bv{\pi}}$.
\item (Learning) At time $T$, apply  \mtt{DL}($\hat{\bv{\epsilon}}$, $\hat{\bv{\delta}}$, $\hat{\bv{\pi}}$) and compute $\gamma^*_T$. If  $\gamma^*_T=\infty$, set $\gamma^*_T=V\log(V)$. Reset $d(T)=0$. 
\item (Control) For $t\geq T$, run \mtt{BISC} with $\hat{\bv{\pi}}$, $\hat{\bv{\epsilon}}$ and $\hat{\bv{\delta}}$, and with effective queue size $\tilde{d}(t) = d(t) + (\gamma^*_T - \theta)^+$. $\Diamond$
\end{achievements} 

Here $\theta$ (to be specified) is a tuning parameter introduced to compensate for the error in  $\gamma^*_T$ (with respect to $\gamma^*$). 
It is interesting to note that \lbisc{} includes three important functions in system control, namely, estimation (data), learning (training) and control (algorithm execution). It also highlights three major sources that contribute to making a system non-intelligent: lack of data samples, incorrect training and parameter tuning, and inefficient control algorithms. 
An intelligent system requires all three to provide good user experience and to be considered smart (Thus, if a search engine does not provide good performance for you at the beginning, it may not be because its algorithm is bad). 













\vspace{-.1in}
\section{Performance Analysis}\label{section:analysis}
In this section, we analyze the performance of \lbisc{}. We focus on three important performance metrics, achieved system intelligence, budget guarantee, and algorithm convergence time. 
\huang{The optimality and convergence analysis is challenging. In particular, the accuracy of the \mle{} estimator affects the quality of dual-learning, which in turn affects algorithm convergence and performance. Thus, the analysis must simultaneously take into account all three components.}

Throughout our analysis, we make the following assumptions, in which we again use $\hat{g}_{\hat{\bv{\pi}}}(\gamma)$ to denote the dual function   in (\ref{eq:dual})  with different $\bv{\epsilon}$,  $\bv{\delta}$ and $\bv{\pi}$ values. 
\begin{assumption}\label{assumption:bdd-LM}
There exists a constant $\nu=\Theta(1)>0$ such that for any valid state distribution $\bv{\pi}' = (\pi'_{1}, ..., \pi'_{H})$ with $\|\bv{\pi}' - \bv{\pi} \|\leq \nu$, 
there exist a set of actions $\{\bv{\mu}^{(h)}_j\}_{j=1, 2, 3}$ with $\bv{\mu}^{(h)}_j\in\script{U}_{z_h}$ and some variables $\{\theta^{(h)}_j\geq0\}_{j=1, 2, 3}$ with $\sum_j\theta^{(h)}_j=1$ for all $z_h$ (possibly depending on $\bv{\pi}'$), such that:
\begin{eqnarray*}
\hspace{-.35in}&&\sum_h\pi'_h\sum_{j=1}^3\theta^{(h)}_j\sum_m [C_m(\mu^{(h)}_{mpj}, z_h) +(1-\mu^{(h)}_{mpj}) a^{(i_h)}_m\overline{C}_m]\leq\rho_0,  
\end{eqnarray*}
where $\rho_0\triangleq\rho-\eta>0$ with $\eta=\Theta(1)>0$ is independent of $\bv{\pi}'$. 
Moreover, for all transition probabilities $\bv{\epsilon}'$ and $\bv{\delta}'$ with $\|\bv{\epsilon}' - \bv{\epsilon} \|\leq \nu$ and $\|\bv{\delta}' - \bv{\delta} \|\leq \nu$, $\rho$ satisfies: 
\begin{eqnarray}
\rho\geq \sum_m \max[\epsilon'_m \overline{C}_m, (1-\delta'_m)\overline{C}_m]. \,\,\Diamond  \label{eq:rho-assumption}
\end{eqnarray}
\end{assumption} 
\vspace{-.15in}
\begin{assumption}\label{assumption:unique} There exists a constant $\nu=\Theta(1)>0$ such that, for any valid state distribution $\bv{\pi}' = (\pi'_{1}, ..., \pi'_{H})$ with $\|\bv{\pi}' - \bv{\pi} \|\leq \nu$,  and transition probabilities $\bv{\epsilon}'$ and $\bv{\delta}'$ with $\|\bv{\epsilon}' - \bv{\epsilon} \|\leq \nu$ and $\|\bv{\delta}' - \bv{\delta} \|\leq \nu$,  $\hat{g}_{\hat{\bv{\pi}}}(\gamma)$ has a unique optimal solution $\gamma^*>\bv{0}$ in $\mathbb{R}$. $\Diamond$
\end{assumption} 
\huang{These two assumptions are standard in the network optimization literature, e.g., \cite{eryilmaz_qbsc_ton07} and \cite{lin-imperfect-scheduling-ton06}. They are necessary conditions that guarantee the budget constraint and are often assumed with $\nu=0$. In our case, having $\nu>0$ means that systems that are alike have similar properties. (\ref{eq:rho-assumption}) is also not restrictive. In fact,  $\sum_m [\pi_{m0}\epsilon'_m \overline{C}_m+ \pi_{m1}(1-\delta'_m)\overline{C}_m]$ ($\pi_{mi}$ is the steady-state probability of being in state $i$ for $m$) is the overall cost without  any pre-service. Hence, (\ref{eq:rho-assumption}) is close to being a necessary condition for feasibility.}  

We now have the third assumption, which is related to the structure of the problem. To state it, we have the following system structural property introduced in \cite{huangneely_dr_tac}. 
\begin{definition}
A system is polyhedral with parameter $\beta>0$ under distribution $\bv{\pi}$ if the dual function $g_{\bv{\pi}}(\gamma)$ satisfies: 
\begin{eqnarray}
g_{\bv{\pi}}(\gamma)\geq g_{\bv{\pi}}(\gamma^*)+\beta\|\gamma^*- \gamma\|.\,\,\,\Diamond\label{eq:polyhedral}
\end{eqnarray}
\end{definition}
\begin{assumption}
\label{assumption:poly} There exists a constant $\nu=\Theta(1)>0$ such that, for any valid state distribution $\bv{\pi}' = (\pi'_{1}, ..., \pi'_{H})$ with $\|\bv{\pi}' - \bv{\pi} \|\leq \nu$,  and transition probabilities $\bv{\epsilon}'$ and $\bv{\delta}'$ with $\|\bv{\epsilon}' - \bv{\epsilon} \|\leq \nu$ and $\|\bv{\delta}' - \bv{\delta} \|\leq \nu$,  $\hat{g}_{\hat{\bv{\pi}}}(\gamma)$ is polyhedral with the same $\beta$. $\Diamond$
\end{assumption}
The polyhedral property often holds for practical systems, especially when   control  action sets are finite (see \cite{huangneely_dr_tac} for more discussions).  

\vspace{-.1in}
\subsection{System Intelligence and Budget}
We first present the performance of \lbisc{} in system intelligence and budget guarantee. 
The following theorem  summarizes the results. 
\begin{theorem}\label{theorem:lbisc-int-budget}
Suppose the system is polyhedral with $\beta=\Theta(1)>0$. By choosing 
$\theta=\max(\frac{V\log(V)^2}{\sqrt{N(T)}},\log(V)^2)$ and a sufficiently large $V$,  with probability at least $1-2Me^{-\log(V)^2/4}$, \lbisc{} achieves: 
\begin{itemize}
\item Budget:  
 \begin{eqnarray}
\hspace{-.3in}&&\tilde{d}(t)\leq d_{\max}\triangleq Vr_d / \hat{D}_{\min} + MC_{\max},\,\forall\,t \label{eq:bdd-d-tilde} \\ 
\hspace{-.3in}&&\overline{d(t)}  = O( \max(\frac{V\log(V)^2}{\sqrt{N(T)}},\log(V)^2) ). \label{eq:bdd-d}
\end{eqnarray}
Here $\overline{d(t)} = \limsup_{t\rightarrow\infty}\frac{1}{t}\sum_{\tau=0}^{t-1}\expect{d(t)}$ and $\hat{D}_{\min}=\Theta(1)$ (defined in (\ref{eq:dmin})). (\ref{eq:bdd-d-tilde}) implies $C_{\text{av}}(\Pi)\leq\rho$. 

\item System intelligence: 
\begin{eqnarray}
\hspace{-.3in}I(\mathtt{LBISC}, \rho)\geq I(\rho) - \frac{B_1+1}{V}  - \max_{z_h}\frac{e_{\max}(\overline{T}^2_h  - \overline{T}_h )}{2V\overline{T}_h}. \label{eq:intelligence-bdd}
\end{eqnarray}
Here $B_1\triangleq B+ 2M(Vr_d +d_{\max} C_{\max})\log(V)/\sqrt{N(T)}$, $e_{\max}\triangleq(MC_{\max}+\rho)^2$, and $\overline{T}_h$ and $\overline{T}^2_h$ are the  first and second moments of return times of  state $z_h$. $\Diamond$ 
\end{itemize}
\end{theorem}
\begin{proof}
See Appendix B. 
\end{proof}

Theorem \ref{theorem:lbisc-int-budget} shows that  \lbisc{} achieves an  $[O(N(T)^{-\frac{1}{2}} + \epsilon), O(\max( N(T)^{-\frac{1}{2}}/\epsilon, \log(1/\epsilon)^2)]$ intelligence-budget tradeoff (taking $\epsilon=1/V$), and the system intelligence level can indeed be pushed arbitrarily close to $I(\rho)$ under \lbisc{}. 
\huang{Thus, by varying the value of $V$, one can tradeoff the intelligence level loss and the budget deficit as needed.}  
Although Theorem \ref{theorem:lbisc-int-budget} appears similar to previous results with learning, e.g., \cite{huang-learning-sig-14}, its analysis is different due to (i) the Markov nature of the demand state $\bv{A}(t)$, and (ii) the learning error in both transition rates in (\ref{eq:action-opt}) and the distribution. 




\vspace{-.1in}
\subsection{Convergence Time}
We now look at the convergence speed of \lbisc, which measures how fast the algorithm learns the desired system operating point. 
%
%
To  present the results, we adopt the following definition of convergence time from \cite{huang-learning-sig-14}  to our setting. 
\begin{definition} \label{def:convergence-time}
Let $\zeta>0$ be a given constant.  The $\zeta$-convergence time of a control algorithm, denoted by $T_{\zeta}$, is the time it takes for  $\tilde{d}(t)$ to get to within $\zeta$ distance of $\gamma^*$, i.e., 
$T_{\zeta}\triangleq\inf\{t: \, |\tilde{d}(t)-\gamma^*|\leq\zeta\}$. $\Diamond$ 
\end{definition}

\huang{The intuition behind Definition \ref{def:convergence-time} is as follows. Since the \lbisc{} algorithm is a queue-based algorithm, the algorithm will starting making optimal choice of actions once $\tilde{d}(t)$ gets close to $\gamma^*$. Hence, $T_{\zeta}$ naturally captures the time it takes for \lbisc{} to converge.} 
We now present our theorem. For comparison, we also analyze the convergence time of \bisc.\footnote{Here the  slower convergence speed of \bisc{} is due to the fact that it does not utilize the information to perform learning-aided control. We present this result to highlight the importance of learning in control.}  
\begin{theorem}\label{theorem:convergence}
Suppose the conditions in Theorem \ref{theorem:lbisc-int-budget} hold. Under \lbisc, with probability at least  $1-2Me^{-\log(V)^2/4}$, 
\begin{eqnarray}
\expect{  T^{\mathtt{LBISC} }_{D_1}  } &=&  O(\max(\frac{V\log(V)^2}{ \sqrt{N(T)}},  \log(V)^2 ))  + T, \\
\expect{ T^{\mathtt{BISC} }_{D_2} } &=& \Theta(V). 
\end{eqnarray}
Here $D_1=O(V\log(V)/\sqrt{N(T)} +D)$ with $D=\Theta(1)$ and $D_2=\Theta(1)$. 
\end{theorem}
\begin{proof}
See Appendix C. 
\end{proof}

Here $D_1$ may be larger than $ D_2$ is due to the fact that \lbisc{} uses inaccurate estimates of $a^{(i)}_m$ for making decisions. 
In the case when $N(T)=T$, we can recover the $O(V^{2/3})$ convergence time results in \cite{huang-learning-sig-14} by choosing $T=V^{2/3}$. 
Theorem \ref{theorem:convergence} also shows that it is possible to achieve  faster convergence if a system has a larger population of users from which it can collect useful samples for learning the target user quickly. It also explicitly quantifies the speedup factor to be proportional to the \emph{square-root} of the user population size, i.e., when $N(T)=N*T$ where $N$ is the number of users, the speedup factor is $\sqrt{N(T)}/\sqrt{T}=\sqrt{N}$. 
 
This result reveals the interesting fact that a big company that has many users naturally has advantage over companies with smaller user populations, since they can collect more useful data  and adapt to a ``smart'' state faster. 




%


\section{Simulation} \label{section:simulation}
We now present simulation results for   \bisc{} and \lbisc. 
We simulate a three-application system ($M=3$) with the following setting. $(r_{1p}, r_{2p}, r_{3p}) = (3, 5, 8)$ and $r_{mc}=1$ for all $m$. Then, we use $\bv{\epsilon}=(0.6, 0.5, 0.3)$ and $\bv{\delta}=(0.2, 0.6, 0.5)$. 
The channel state space is $\script{S}=\{1, 2\}$ for all $m$, with $\prob{S_1(t)=1}=0.5$, $\prob{S_2(t)=1}=0.3$, and $\prob{S_3(t)=1}=0.3$. The service cost is given by $C_m(1, \bv{S}(t))=S_m(t)$. 
We simulate the system for $T_{sim}=10^5$ slots, with $V=\{5, 10, 20, 50, 100 \}$. For \lbisc, we simulate a user population effect function as in (\ref{eq:user-population}), i.e., $N(T) = f(\#\, \text{of user})\cdot T$ and choose $f(\#\, \text{of user})=2, 5, 8$. 
%
We also fix the value $\rho=3.5$ and choose the learning time $T=V^{2/3}$. 

We first present Fig. \ref{fig:I_versus_rho} that shows  $I(\rho)$ as a function of $\rho$. 
For comparison, we include a second setting, where we change $(r_{1p}, r_{2p}, r_{3p}) = (4, 5, 3)$, $\bv{\epsilon}=(0.8, 0.4, 0.3)$, $\bv{\delta}=(0.2, 0.9, 0.5)$, and $\prob{S_2(t)=1}=0.8$. 
It can be seen that $I(\rho)$ first increases as $\rho$ increases. Eventually $\rho$ becomes more than needed after all the possible predictability has been exploited. 
Then, $I(\rho)$ becomes flat. This  \emph{diminishing return} property   is consistent with our understanding obtained  in Section \ref{section:characterize}.  
\begin{figure}[ht] 
\centering
\includegraphics[height=1.3in, width=3.2in]{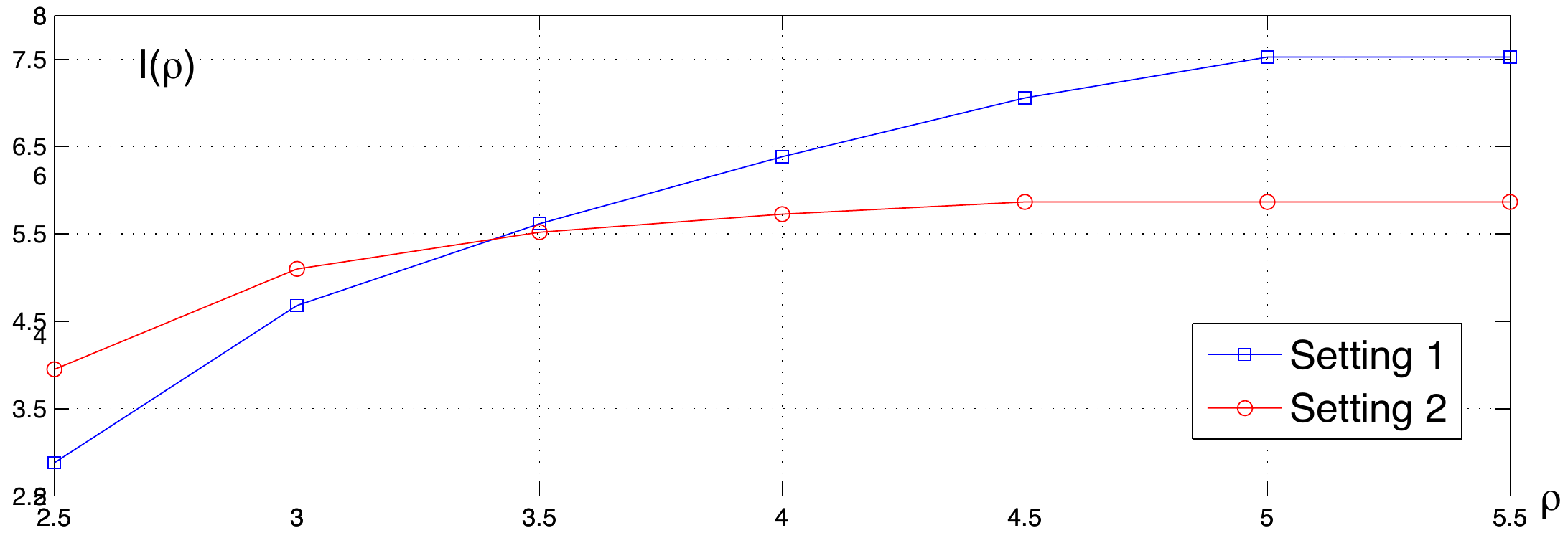}
\vspace{-.1in} 
\caption{$I(\rho)$ versus $\rho$:  in the two settings tested, $I(\rho)$ are both   concave increasing in $\rho$.} 
\label{fig:I_versus_rho}
\vspace{-.1in}
\end{figure}
 
We now look at algorithm performance. From Fig. \ref{fig:int-deficit} we see that both \bisc{} and \lbisc{} are able to achieve  high intelligence levels. 
Moreover,   \lbisc{} does much better in controlling the deficit ($2\times$-$4\times$ saving compared to \bisc). We remark here that \bisc{} assumes full knowledge beforehand, while \lbisc{} learns them online. 
\begin{figure}[ht] 
\centering
\includegraphics[height=1.4in, width=3.3in]{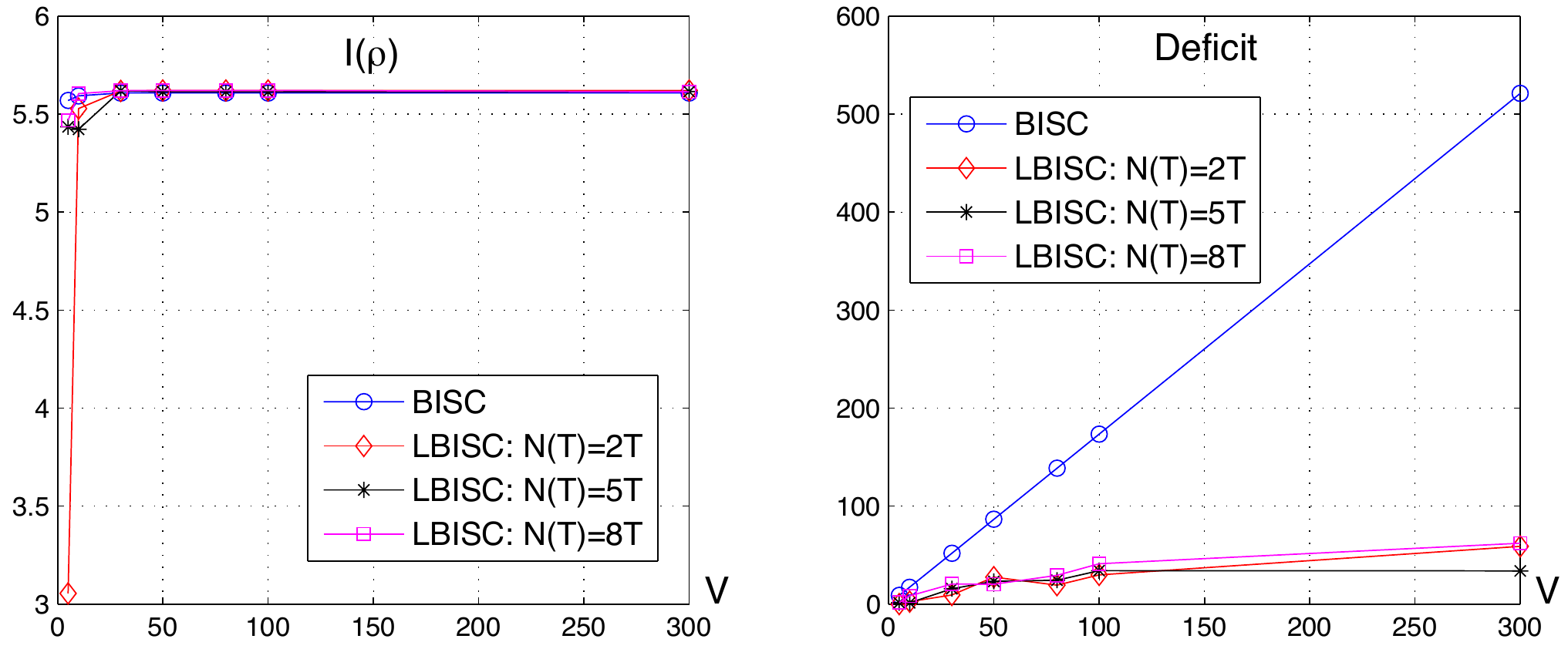}
\vspace{-.1in} 
\caption{Intelligence and deficit performance of \bisc{} and \lbisc{} with different user-population effect.} 
\label{fig:int-deficit}
\vspace{-.1in}
\end{figure}

Finally, we look at the convergence properties of the algorithms. Fig. \ref{fig:sample-path} compares \bisc{} and \lbisc{} with $N(t)=8T$ and $V=300$. We see that \lbisc{} converges at around $460$ slots, whereas \bisc{} converges at around $920$ slots, resulting in a $2\times$ improvement. Moreover, the actual deficit level under \lbisc{} is much smaller compared to that under \bisc{} ($80$ versus $510$, a $6\times$ improvement). 
From this result, we see that it is important to efficiently utilize the data samples collected over time, and dual learning provides one way to boost algorithm convergence. 
\begin{figure}[ht] 
\centering
\vspace{-.1in}
\includegraphics[height=1.5in, width=3.3in]{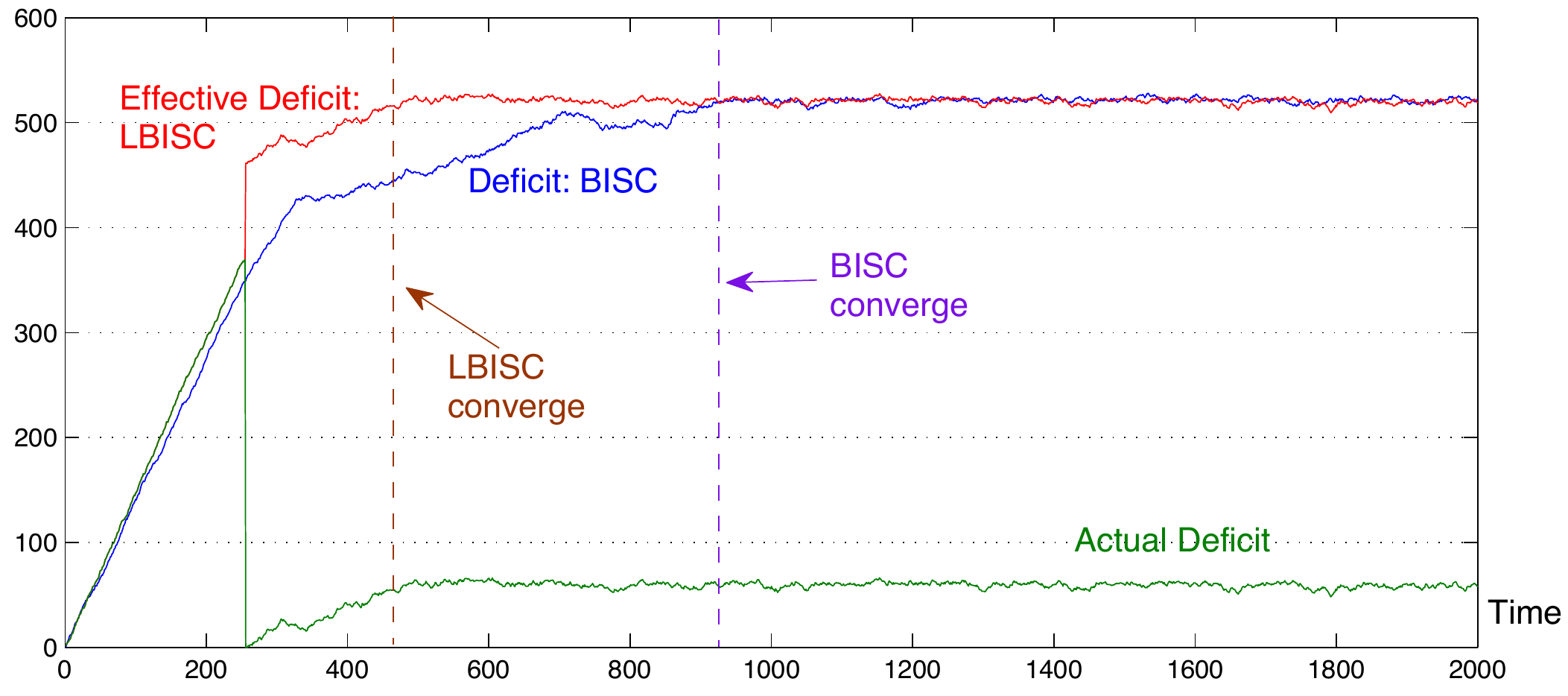}
\vspace{-.1in} 
\caption{Convergence of \bisc{} and \lbisc{} with $N(T)=8T$ for $V=300$.} 
\label{fig:sample-path}
\vspace{-.1in}
\end{figure}


\vspace{-.1in}
\section{Conclusion}\label{section:conclusion} 
In this paper, we present a general framework for defining and understanding system intelligence, and  propose a novel metric for measuring the smartness of dynamical systems,  defined to be the maximum average reward rate obtained by proactively serving user demand subject to a resource constraint.  
We show that the highest system intelligence level is jointly determined by system resource, action costs, user demand volume, and correlation among demands. We then develop a learning-aided algorithm called Learning-aided Budget-limited Intelligent System Control (\mtt{LBISC}), which efficiently utilizes data samples about system dynamics and achieves near-optimal intelligence, and guarantees a deterministic deficit bound. Moreover, \lbisc{} converges much faster compared to its non-learning based counterpart. 

$\vspace{-.1in}$
\bibliographystyle{unsrt}
\bibliography{mybib}

\vspace{-.1in}
\section*{Appendix A -- Proof of Theorem \ref{theorem:upper-bdd}}
We prove Theorem \ref{theorem:upper-bdd} here. 
\begin{proof} (Theorem \ref{theorem:upper-bdd}) 
Consider any control scheme $\Pi$ and fix a time $T$. 
Define the joint state $z(t)=(\bv{A}(t), \bv{S}(t))$ and denote its state space as $\script{Z}=\{z_1, ..., z_H\}$. 
Then, consider a state $z_h$ and let $\script{T}_h(T)$ be the set of slots with $z(t)=z_h$ for $t=1, ..., T$. 
Denote $\{ \bv{\mu}_{p}(0), ..., \bv{\mu}_{p}(T)\}$ the pre-service decisions made by $\Pi$. Denote the following joint reward-cost pair:\footnote{\huang{To save space, here we use the notation $\expect{X; Y\left.|\right.A}$ to denote $(\expect{X\left.|\right.A}, \expect{Y\left.|\right.A})$.}}  
\begin{eqnarray}
\hspace{-.3in}&&(Reward^{(h)}(T), Cost^{(h)}(T)) \triangleq \nonumber\\
\hspace{-.3in}&&   \frac{1}{T}\sum_{\tau=0}^{T-1}\sum_m\expect{ I_{[A_m(\tau+1)=1]} [\mu_{mp}(\tau)r_{mp} + (1-\mu_{mp}(\tau))r_{mc}];  \nonumber\\
\hspace{-.3in}&&  C_m(\mu_{mp}(\tau), z_h) + (1-\mu_{mp}(\tau))I_{[A_m(\tau+1)=1]}\overline{C}_m \left.|\right. z(\tau)=z_h   }. \nonumber
\end{eqnarray}

Notice that this is a mapping from $z_h$ to a subset in $\mathbb{R}^2$ and that both the reward and the cost are continuous. 
Also note that $\expect{ I_{[A_m(\tau+1)=1]}} = a^{(i_h)}_m$, where $i_h=A_m^{(h)}$  is defined in (\ref{eq:effective-reward-def}) to denote the probability of having $A_m(\tau+1)=1$ given $A_m(\tau)=A_m^{(h)}$. Using the independence of $\bv{A}(t)$ and $\bv{S}(t)$, and Caratheodory's theorem \cite{bertsekasoptbook}, it follows that there exists three vectors $\bv{\mu}^{(h)}_j(T)\in\script{U}_{z_h}, j=1, 2, 3$, with appropriate weights $\theta^{(h)}_j(T)\geq0, j=1, 2, 3$, and $\sum_i \theta^{(h)}_j(T)=1$, so that: 
\begin{eqnarray}
\hspace{-.3in}&&(Reward^{(h)}(T), Cost^{(h)}(T))  \\
\hspace{-.3in}&&\triangleq \sum_{j=1}^3\theta^{(h)}_j(T) \sum_m \big(a^{(i_h)}_m [\mu^{(h)}_{mpj}(T)r_{mp} + (1- \mu^{(h)}_{mpj}(T))r_{mc}]; \nonumber \\
\hspace{-.3in}&&\qquad\qquad\quad\quad C_m(\mu^{(h)}_{mpj}(T), z_h)     +(1-\mu^{(h)}_{mpj}(T)) a^{(i_h)}_m\overline{C}_m\big). \nonumber
\end{eqnarray}
Now consider averaging the above over all $z_h$ states.  
We get: 
\begin{eqnarray}
\hspace{-.3in}&&(Reward_{av}(T), Cost_{av}(T))\triangleq   \nonumber \\
\hspace{-.3in}&&\sum_h\pi_h\sum_{j=1}^3\theta^{(h)}_j(T)\sum_m \big(a^{(i_h)}_m [\mu^{(h)}_{mpj}(T)r_{mp}  + (1- \mu^{(h)}_{mpj}(T))r_{mc}];\nonumber \\
\hspace{-.3in}&&\qquad\qquad\qquad   C_m(\mu^{(h)}_{mpj}(T), z_h)     +(1-\mu^{(h)}_{mpj}(T)) a^{(i_h)}_m\overline{C}_m\big). \nonumber
\end{eqnarray}
Using a similar argument as in the proof of Theorem 1 in \cite{neelyenergy}, one can show that there exist limit points $\theta^{(h)}_j\geq0$ and $\bv{\mu}^{(h)}_j$ as $T\rightarrow\infty$, so that the reward-cost tuple can be expressed as: 
\begin{eqnarray}
\hspace{-.3in}&&(Reward_{av}, Cost_{av}) \nonumber \\
\hspace{-.3in}&&=  \sum_h\pi_h\sum_{j=1}^3\theta^{(h)}_j\sum_m \big(a^{(i_h)}_m [\mu^{(h)}_{mpj}r_{mp}  + (1- \mu^{(h)}_{mpj})r_{mc}];\nonumber \\
\hspace{-.3in}&&\qquad\qquad\qquad\qquad   C_m(\mu^{(h)}_{mpj}, z_h)     +(1-\mu^{(h)}_{mpj}) a^{(i_h)}_m\overline{C}_m\big). \nonumber
\end{eqnarray} 
This shows that for an arbitrarily control algorithm $\Pi$, its average reward and cost can be expressed as those in problem (\ref{eq:bdd-obj}). Hence, its budget limited average reward cannot exceed $\Phi$, which is the optimal value of (\ref{eq:bdd-obj}). This shows that $\Phi\geq I(\rho)$. 

The other direction $I(\rho)\geq\Phi$ will be shown in the analysis of the \lbisc{} algorithm, where we show that \lbisc{} achieves an intelligence level arbitrarily close to $\Phi$. 
\end{proof}

\section*{Appendix B -- Proof of Theorem \ref{theorem:lbisc-int-budget}}
We prove Theorem \ref{theorem:lbisc-int-budget} here with the following two lemmas, whose proofs are given in Appendix D. 
The first lemma bounds the error in estimating $\epsilon_m$ and $\delta_m$. 
\begin{lemma}\label{lemma:mle-bdd}
\mtt{MLE}($T$) ensures that: 
\begin{eqnarray}
\prob{\max_m|\hat{\epsilon}_m - {\epsilon}_m, \hat{\delta}_m - {\delta}_m|\geq \frac{\log(V)}{\sqrt{N(T)}}}\leq 2Me^{-\log(V)^2/4}.\label{eq:error-bdd} 
\end{eqnarray}
\end{lemma}
The second lemma bounds the error in estimating $\gamma^*$. In the lemma, we define an intermediate dual function  
$\hat{g}_{\bv{\pi}}(\gamma)$ defined as: 
\begin{eqnarray}
\hspace{-.2in}&& \hat{g}_{\bv{\pi}}(\gamma) \triangleq \sum_h \pi_h\sum_m\sup_{\bv{\mu}^{(h)}_p}\bigg\{ V\hat{a}^{(i_h)}_m[\mu^{(h)}_{mp}r_{mp}  + (1- \mu^{(h)}_{mp})r_{mc}] \nonumber \\
\hspace{-.2in}&&\qquad\qquad - \gamma  [C_m(\mu^{(h)}_{mp}, z_h)      +(1-\mu^{(h)}_{mp}) \hat{a}^{(i_h)}_m\overline{C}_m - \rho]\bigg\}. \label{eq:dual-tilde}
\end{eqnarray}
That is, $\hat{g}_{\bv{\pi}}(\gamma)$ is the dual function defined with the true distribution $\pi_h$ and the estimated $\hat{a}^{(i_h)}_m$.  We use $\hat{\gamma}^*$ to denote the optimal solution of $\hat{g}_{\bv{\pi}}(\gamma)$. 
\begin{lemma} \label{lemma:multiplier-bdd}
With probability at least $1-2Me^{-\log(V)^2/4}$, \mtt{DL} outputs a $\gamma^*_T$ that satisfies $| \gamma^*_T - \gamma^*  | \leq d_{\gamma}$ where $d_{\gamma}\triangleq \frac{cV\log(V)}{\sqrt{N(T)}}$ and  $c=\Theta(1)>0$. Moreover, $| \gamma^* - \hat{\gamma}^*  | \leq d_{\gamma}$. 
\end{lemma}

We now prove Theorem \ref{theorem:lbisc-int-budget}. 
\begin{proof} (Theorem \ref{theorem:lbisc-int-budget}) We first prove the budget bound, followed by the intelligence performance. 

(\textbf{Budget}) We want to show that under \lbisc, if $\gamma^*_T$ is estimated accurate enough (happens with high probability), $\tilde{d}(t)$ is deterministically bounded throughout. This will imply that  the budget constraint is met. 
Note that under \lbisc, $D_m(t)$ in (\ref{eq:dm-def}) is defined with $\hat{a}^{(i)}_m$. Thus, we denote it as $\hat{D}_m(t)$. 

To see this, first note from Lemma \ref{lemma:mle-bdd} that with probability at least $1-2Me^{-\log(V)^2/4}$, $\max_m|\hat{\epsilon}_m - {\epsilon}_m, \hat{\delta}_m - {\delta}_m|\leq \frac{\log(V)}{\sqrt{N(T)}}$. Hence, $\|\bv{\epsilon} - \hat{\epsilon}\|\leq \nu$ and $\|\bv{\delta} - \hat{\delta}\|\leq \nu$ when $V$ is large. Thus, (\ref{eq:rho-assumption}) in Assumption \ref{assumption:bdd-LM} holds. 
Denote 
\begin{eqnarray}
\hspace{-.2in}&&\hat{D}_{\min} \triangleq \label{eq:dmin}\\
\hspace{-.2in}&& \qquad\min_{m, k, i}\{  C_m(1, s_k) -  \hat{a}^{(i)}_m\overline{C}_m:  C_m(1, s_k) -  \hat{a}^{(i)}_m\overline{C}_m>0\}, \nonumber
\end{eqnarray}
and  look at the algorithm steps (\ref{eq:dm-def}) and (\ref{eq:action-opt}). We claim that whenever $\tilde{d}(t)\geq Vr_d / \hat{D}_{\min}$, it stops increasing. 
To see this, let us fix an $m$ and consider two cases. 

(i) Suppose in (\ref{eq:dm-def}) $\hat{D}_m(t)>0$ (defined with $\hat{a}^{(i)}_m$), then we must have $\mu_{mp}(t)=0$, as $V\hat{a}_m^{(i)} (r_{mp} - r_{mc}) -  d(t) \hat{D}_m(t)\leq 0$ in (\ref{eq:action-opt}) by definition of $\hat{D}_{\min}$. 

(ii) Suppose instead $\hat{D}_m(t)\leq0$. Although in this case we set $\mu_{mp}(t)=1$, it also implies that the state $s_k$ is such that $C_m(1, s_k)\leq \hat{a}_m^{(i)}\overline{C}_m$ (from the definition in (\ref{eq:dm-def})). 
Combing the two cases, we see that whenever $\tilde{d}(t)\geq Vr_d / \hat{D}_{\min}$, $\tilde{C}(t)\leq \sum_m  \hat{a}_m^{(i)}\overline{C}_m$, which implies that $\tilde{C}(t)\leq\rho$ by (\ref{eq:rho-assumption}) in Assumption \ref{assumption:bdd-LM}. Therefore, we conclude that: 
\begin{eqnarray}
\tilde{d}(t)\leq d_{\max}\triangleq Vr_d / \hat{D}_{\min} + MC_{\max}. \label{eq:bdd-d-tilde2}
\end{eqnarray}
This completes the proof of (\ref{eq:bdd-d-tilde}). The proof for (\ref{eq:bdd-d}) is given in the intelligence part at (\ref{eq:bdd-d-deviation}). 



(\textbf{Intelligence}) First, we add to both sides of (\ref{eq:drift-ineq-1}) the \emph{drift-augmentation} term $\Delta_a(t)\triangleq-(\gamma^*_T - \theta)^+[\rho - \tilde{C}(\bv{\mu}(t))]$ and obtain: 
\begin{eqnarray}
\hspace{-.3in}&&\Delta(t) +\Delta_a(t) - V\tilde{r}(\bv{\mu}(t)) \label{eq:drift-ineq-2} \\
\hspace{-.3in}&&\qquad\qquad \leq B - \bigg(V\tilde{r}(\bv{\mu}(t)) + \tilde{d}(t) [\rho-\tilde{C}(\bv{\mu}(t)) ]\bigg). \nonumber
\end{eqnarray} 
From Lemma \ref{lemma:multiplier-bdd}, we know that the event $\{| \gamma^*_T - \gamma^*  | \leq d_{\gamma} \}$ takes place with probability at least $1-2Me^{-\log(V)^2/4}$. Hence, from now on, we carry out our argument conditioning on this event. 

Note that in every time, we use the estimated $\hat{\bv{\epsilon}}$ and $\hat{\bv{\delta}}$ for decision making, i.e., we maximize: 
\begin{eqnarray*} 
\hspace{-.3in}&&   \sum_{m} \mu_{mp}(t)[V[a_m^{(i)} + e(a_m^{(i)})] (r_{mp} - r_{mc}) \\
\hspace{-.3in}&&   \qquad\qquad -  d(t)(C_m(1, \bv{S}(t)) - [a_m^{(i)} + e(a_m^{(i)})]\overline{C}_m)],  
\end{eqnarray*}
where $e(a_m^{(i)}) =\hat{a}_m^{(i)} - a_m^{(i)}$. Let $\bv{\mu}_p^{L}(t)$ be the action chosen by \lbisc, and let $\bv{\mu}_p^{*}(t)$ be the actions chosen by \bisc, i.e. with the exact $a_m^{(i)}$ values.  
We see then: 
\begin{eqnarray*}
\hspace{-.3in}&& \sum_{m} \mu^L_{mp}(t)V[\hat{a}_m^{(i)}  (r_{mp} - r_{mc}) -  d(t) \hat{D}_m(t)] \\
\hspace{-.3in}&&\qquad\qquad\geq  \sum_{m} \mu^*_{mp}(t)V[\hat{a}_m^{(i)}  (r_{mp} - r_{mc}) -  d(t) \hat{D}_m(t)]. 
\end{eqnarray*}
With the definition of $\hat{D}_m(t)$ and that $\max_m|\hat{\epsilon}_m - {\epsilon}_m, \hat{\delta}_m - {\delta}_m|\leq \frac{\log(V)}{\sqrt{N(T)}}$ in Lemma \ref{lemma:mle-bdd}, this implies that: 
\begin{eqnarray*}
\hspace{-.3in}&& \sum_{m} \mu^L_{mp}(t)V[a_m^{(i)}  (r_{mp} - r_{mc}) -  d(t) D_m(t)] \\
\hspace{-.3in}&&\qquad \geq  \sum_{m} \mu^*_{mp}(t)V[a_m^{(i)}  (r_{mp} - r_{mc}) -  d(t) D_m(t)] - E_{tot}, 
\end{eqnarray*}
where $E_{tot}\triangleq 2M(Vr_d +d_{\max} C_{\max})\frac{\log(V)}{\sqrt{N(T)}}$. This means that the actions chosen by \lbisc{} approximately maximize (\ref{eq:action-opt}). 
Therefore, by comparing the term in $V\tilde{r}(\bv{\mu}(t)) + \tilde{d}(t) [\rho-\tilde{C}(\bv{\mu}(t)) ]$ in (\ref{eq:drift-ineq-2}) and the definition of $g_h(\gamma)$ in (\ref{eq:dual-state}), we have: 
\begin{eqnarray} \Delta(t) +\Delta_a(t) - V\tilde{r}(\bv{\mu}(t))   \leq B_1 - g_{z(t)}(\tilde{d}(t)). \nonumber
\end{eqnarray}  
Here $B_1\triangleq B+E_{tot}$ and $B\triangleq \rho_{\max}^2+M^2C_{\max}^2$. 

Now assume without loss of generality that $z(0)=z$ and let $t_n$ be the $n$-th return time for $z(t)$ to visit $z$.\footnote{$z(t)$ is a finite-state irreducible and aperiodic Markov chain. Hence, the return time is well defined for each state.}  We get: 
\begin{eqnarray*}
\hspace{-.2in}&&\sum_{t=t_n}^{t_{n+1}-1}[\Delta(t) +\Delta_a(t) - V\tilde{r}(\bv{\mu}(t))]  \\
\hspace{-.2in}&&   \leq B_1 (t_{n+1}-t_n) -\sum_{t=t_n}^{t_{n+1}-1} g_{z(t)}(\tilde{d}(t)) \nonumber\\
\hspace{-.2in}&& \stackrel{(*)}{\leq} B_1 (t_{n+1}-t_n) -\sum_{t=t_n}^{t_{n+1}-1} g_{z(t)}(\tilde{d}(t_n)) +\sum_{t=t_n}^{t_{n+1}-1} (t-t_n)e_{\max} \\
\hspace{-.2in}&&   \leq B_1 (t_{n+1}-t_n) -\sum_{t=t_n}^{t_{n+1}-1} g_{z(t)}(\tilde{d}(t_n)) +\frac{T_n^2  - T_n}{2}e_{\max}. 
\end{eqnarray*} 
Here $e_{\max} = (MC_{\max}+\rho)^2$ and $T_n\triangleq t_{n+1}-t_n$ denotes the length of the $n$-th return period, and (*) follows since $|g_{z(t)}(\tilde{d}(t_n)) - g_{z(t)}(\tilde{d}(t))|\leq |\tilde{d}(t_n) - \tilde{d}(t) | (MC_{\max}+\rho)$, and $|\tilde{d}(t_n) - \tilde{d}(t) |\leq (t-t_n)(MC_{\max}+\rho)$. 

Taking an expectation over  $T_n$,  and using that for any state $z_h$, the expected time to visit $z_h$ during the return period to $z$ is $\pi_{h}/\pi_{z}$ and that $\overline{T}_z=1/\pi_{z}$ \cite{aldousmarkovchainbook}, we get: 
\begin{eqnarray*}
\hspace{-.3in}&&\expect{\sum_{t=t_n}^{t_{n+1}-1}[\Delta(t) +\Delta_a(t) - V\tilde{r}(\bv{\mu}(t))]  \left.|\right. z(t_n)=z, \tilde{d}(t_n) }  \\
\hspace{-.3in}&&\qquad\quad   \leq B_1\overline{T}_z - \overline{T}_z\sum_{h}\pi_{h} g_{z(t)}(\tilde{d}(t_n)) +\frac{\overline{T}_z^2  - \overline{T}_z}{2}e_{\max}\\
\hspace{-.3in}&&\qquad\quad \leq B_1\overline{T}_z - V\overline{T}_zI(\rho) +\frac{\overline{T}_z^2  - \overline{T}_z}{2}e_{\max}
\end{eqnarray*} 
Here in the last step we have used $g_{\bv{\pi}}(\gamma)\geq g_{\bv{\pi}}(\gamma^*) \geq VI(\rho)$ in  (\ref{eq:dual-primal}). Therefore, taking an expectation over $\tilde{d}(t_n)$ and  taking a telescoping sum over $\{t_0, t_1, ...t_n\}$, and using an argument almost identical as in the proof of Theorem 2 in \cite{huangneely_qlamarkovian}, we obtain: 
\begin{eqnarray}
\hspace{-.3in}&&\liminf_{n\rightarrow\infty}\frac{1}{n}\sum_{t=0}^{ n-1}\expect{\tilde{r}(t) \left.|\right. z(0)=z} \geq I(\rho) - B_1/V \label{eq:intermediate} \\
\hspace{-.3in}&&\quad - \frac{e_{\max}(\overline{T}_z^2  - \overline{T}_z )}{2V\overline{T}_z} + \frac{1}{V}\liminf_{n\rightarrow\infty}\frac{1}{n}\sum_{t=0}^{n-1}\expect{\Delta_a(t)\left.|\right. z(0)=z}. \nonumber  
\end{eqnarray}
It remains to show that the last term is small. To do so, recall its definition $\Delta_a(t)=-(\gamma^*_T - \theta)^+[\rho - \tilde{C}(\bv{\mu}(t))]$. According to \lbisc, $\gamma^*_T$ is fixed at time $T$. Hence, we have (from now on we drop the conditioning on $z(0)=z$ for brevity): 
\begin{eqnarray*}
\hspace{-.3in}&&\liminf_{n\rightarrow\infty}\frac{1}{n}\sum_{t=0}^{n-1}\expect{\Delta_a(t)} \\
\hspace{-.3in}&& = - (\gamma^*_T - \theta)^+\liminf_{n\rightarrow\infty}\frac{1}{n}\sum_{t=0}^{n-1}\expect{\rho - \tilde{C}(\bv{\mu}(t))}
\end{eqnarray*}
We thus want to show that this term is $O(1)$. To do so, note that $\rho$ and $\tilde{C}(\bv{\mu}(t))$ are the service and arrival rates to the \emph{actual} queue $d(t)$. From the queueing dynamics we have: 
\begin{eqnarray}
&& \liminf_{n\rightarrow\infty}\frac{1}{n}\sum_{t=0}^{n-1}\expect{\rho - \tilde{C}(\bv{\mu}(t)) }\label{eq:prob-bdd}\\ 
&&\qquad\leq \liminf_{n\rightarrow\infty}\frac{1}{n}\sum_{t=0}^{n-1} \expect{1_{d(t)<\rho}}(\rho+C_{\max})\nonumber \\
&&\qquad= \liminf_{n\rightarrow\infty}\frac{1}{n}\sum_{t=0}^{n-1} \prob{  d(t)<\rho }(\rho+C_{\max}). \nonumber
\end{eqnarray}
Hence, it remains to show that $\liminf_{n\rightarrow\infty}\frac{1}{n}\sum_{t=0}^{n-1} \prob{  d(t)<\rho }$ is small. 

Recall that  $\hat{\gamma}^*$ is the optimal solution for $\hat{g}_{\bv{\pi}}(\gamma)$ defined in (\ref{eq:dual-tilde}), and that $\hat{g}_{\bv{\pi}}(\gamma)$  is polyhedral with parameter $\beta=\Theta(1)$ according to Assumption \ref{assumption:poly}. 
Hence, Theorem 1 in \cite{huangneely_dr_tac} shows that there exist $\Theta(1)$ constants $N_1$, $D$ and $\eta_1>0$, such that at every time $t$, if $|\tilde{d}(t) - \hat{\gamma}^*|\geq D$, 
\begin{eqnarray}
\expect{|\tilde{d}(t+N_1) - \hat{\gamma}^*| \left.|\right. \tilde{d}(t)} \leq |\tilde{d}(t) - \hat{\gamma}^*| - \eta_1. \label{eq:t1-drift} 
\end{eqnarray}
Using $\theta = \max(\frac{V\log(V)^2}{\sqrt{N(T)}},\log(V)^2)$, we see that when $V$ is large, $\theta/2\geq d_{\gamma} = \frac{cV\log(V)}{\sqrt{N(T)}}$. Combining it with   $\tilde{d}(t) = d(t) +\gamma^*_T - \theta$  (conditioning on $\{| \gamma^*_T - \gamma^*  | \leq d_{\gamma}\}$, we have $(\gamma^*_T - \theta)^+=\gamma^*_T - \theta$) and  $ \hat{\theta}\triangleq\hat{\gamma}^* - \gamma^*_T+\theta\in [\theta/2, \theta]$,  the above implies that whenever $|d(t) - \hat{\theta} |\geq D$, 
\begin{eqnarray}
\expect{|d(t+N_1) - \hat{\theta} | \left.|\right. \tilde{d}(t)} \leq |d(t) - \hat{\theta}| - \eta_1. 
\end{eqnarray}
Using the same proof argument in Theorem 1 in \cite{huangneely_dr_tac}, we can show that there exist $\Theta(1)$ positive constants $c_1$ and $c_2$ that: %
\begin{eqnarray}
\limsup_{t\rightarrow\infty}\frac{1}{t}\sum_{\tau=0}^{t-1}\prob{ |d(t)- \hat{\theta}| >D + l}\leq c_1e^{-c_2l}. \label{eq:bdd-d-deviation}
\end{eqnarray}
This shows that $\overline{d(t)} =\Theta(\hat{\theta}) = O( \max(\frac{V\log(V)^2}{\sqrt{N(T)}},\log(V)^2) )$. 

Since $D=\Theta(1)$ and $\theta\geq\log(V)^2$, it can be seen that when $V$ is large, 
\begin{eqnarray}
\hspace{-.3in}&& \liminf_{n\rightarrow\infty}\frac{1}{n}\sum_{t=0}^{n-1}\expect{\rho - \tilde{C}(\bv{\mu}(t)) }\label{eq:prob-bdd2}\\ 
\hspace{-.3in}&& \leq \liminf_{n\rightarrow\infty}\frac{1}{n}\sum_{t=0}^{n-1} \prob{ | d(t) - \hat{\theta}|>D+\frac{2\log(V)}{c_2}}(\rho+C_{\max}) \nonumber \\
\hspace{-.3in}&& = O(1/V^2). 
\end{eqnarray}
Combining with the fact that $\gamma^*_T-\theta=\Theta(V)$, we conclude that the last term in (\ref{eq:intermediate}) is $O(1/V)$. This completes the proof of (\ref{eq:intelligence-bdd}). 
\end{proof}

\section*{Appendix C -- Proof of Theorem \ref{theorem:convergence}}
We will use of the following technical lemma from \cite{bertsekasoptbook} for our proof.
\begin{lemma}\label{lemma:exp-time} \cite{bertsekasoptbook} 
Let $\script{F}_n$ be filtration, i.e., a sequence of increasing $\sigma$-algebras with $\script{F}_n\subset\script{F}_{n+1}$. Suppose the sequence of random variables $\{y_n\}_{n\geq0}$ satisfy:
\begin{eqnarray} 
\expect{||y_{n+1}-y^*||\left.|\right. \script{F}_n} \leq \expect{||y_{n}-y^*|| \left.|\right. \script{F}_n} -u_n, \label{eq:exp-dec}
\end{eqnarray}
where $u_n$ takes the following values:
\begin{eqnarray}
u_n=\left\{\begin{array}{ll} u & \textrm{if  $||y_{n}-y^*||\geq D$}, \\0 &\textrm{else}.\end{array}\right.\label{eq:eta-n}
\end{eqnarray}
Here $u>0$ is a given constant.
 Then, by defining $N_{D}\triangleq\inf\{k \left.|\right. \|y_{n}-y^*\| \leq D\}$, we have:
\begin{eqnarray}
\expect{N_D} \leq ||y_{0}-y^*|| /u. \,\, \Diamond
\end{eqnarray}
\end{lemma} 

We now present the proof. 
\begin{proof} (Theorem \ref{theorem:convergence}) 
To start,  note that the first $T$ slots are spent learning $\hat{\epsilon}$, $\bv{\delta}$, and $\hat{\bv{\pi}}$. Lemma \ref{lemma:multiplier-bdd} shows that  after $T$ slots, with high probability, we have $| \gamma^*_T - \gamma^*  | \leq \frac{cV\log(V)}{\sqrt{N(T)}}$ for some $c=\Theta(1)$.  Using the definition of $\tilde{d}(t)$, this implies that when $V$ is large, 
\begin{eqnarray}
&& |\tilde{d}(T) - \gamma^*| \leq \theta/2 \\
&& = \max(V\log(V)^2/\sqrt{N(T)}, \log(V)^2)/2. \nonumber
\end{eqnarray}
Using (\ref{eq:t1-drift}) in the proof of Theorem \ref{theorem:lbisc-int-budget}, and applying  Lemma \ref{lemma:exp-time}, we see then the expected time for $\tilde{d}(t)$ to get to within $D$ of $\hat{\gamma}^*$, denoted by $\tilde{T}_D$, satisfies: 
\begin{eqnarray}
\expect{ \tilde{T}_D } \leq N_1\max(V\log(V)^2/\sqrt{N(T)}, \log(V)^2)/2\eta_1. 
\end{eqnarray}
Here recall that $N_1=\Theta(1)$ is a constant in (\ref{eq:t1-drift}).  

Since $|\gamma^* -\hat{\gamma}^*|\leq  d_{\gamma}=cV\log(V)/\sqrt{N(T)}$, by defining $D_1 = cV\log(V)/\sqrt{N(T)} +D$, we conclude that: 
\begin{eqnarray}
\expect{ T^{\mathtt{LBISC}}_{D_1} } \leq N_1\max(\frac{V\log(V)^2}{2\eta_1\sqrt{N(T)}}, \frac{\log(V)^2}{2\eta_1})  + T.  
\end{eqnarray}

In the case of \bisc, one can similarly show that there exists $\Theta(1)$ constants $N_1$,  $D$ (only related to $z(t)$ and $\beta$), and $\eta_2$, so that, 
\begin{eqnarray}
\expect{|d(t+N_1) - \gamma^*| \left.|\right. d(t)} \leq |d(t) - \gamma^*| - \eta_2. \label{eq:t1-drift2} 
\end{eqnarray}
Since $\gamma^*=\Theta(V)$ \cite{huangneely_dr_tac} and $d(0)=0$, we conclude that: 
\begin{eqnarray}
\expect{ T^{\mathtt{LBISC}}_{D} } \leq N_1V/\eta_2.  
\end{eqnarray}
This completes the proof of the theorem. 
\end{proof}

\section*{Appendix D -- Proof of Lemmas \ref{lemma:mle-bdd} and \ref{lemma:multiplier-bdd}} 
We will make use of the following results from \cite{gillman-mc-deviation-93}
\begin{theorem}\label{theorem:mc-deviation} \cite{gillman-mc-deviation-93}. 
For a finite state irreducible and aperiodic and reversible Markov chain with state space with state space $\script{G}$ 
steady-state distribution $\bv{\pi}$ and an initial distribution $\bv{q}$. Let $y=\max_{s_1, s_2\in\script{G}}\frac{\pi_{s_1}}{\pi_{s_2}}$ and let $N_q=\|\frac{q_{s}}{\pi_{s}},\, s\in\script{G}  \|_2$. Moreover, let $a=1-\lambda_2$ where $\lambda_2<1$  is the second largest eigenvalue of the transition matrix. Then, for any subset $A\subset\script{G}$ and let $t_n$ be the number of visits to $A$ in $n$ steps. We have for any $b$ that: 
\begin{eqnarray}
\prob{|t_n - n\pi_{A}|\geq b} \leq 2N_qe^{-b^2a/20ny }. \label{eq:prob-deviation}
\end{eqnarray}
\end{theorem}
\begin{proof} (Lemma \ref{lemma:mle-bdd}) 
We first estimate the number of times state $A_m(t)=0$ is visited, denoted by $N_{m0}$. Using Theorem \ref{theorem:mc-deviation}, with $b=\sqrt{N(T)}\log(V)$ and using the fact that both $N_q$, $y$ and $a>0$ are $\Theta(1)$, we have:  
\begin{eqnarray}
\hspace{-.3in}&&\prob{|N_{m0}-N(T)\pi_{m0}| \geq \sqrt{N(T)}\log(V)} \nonumber \\
\hspace{-.3in}&&\qquad\qquad\qquad\qquad\qquad\qquad \leq c_{m1}e^{-c_{m2}\log(V)^2}. \label{eq:pi-0-bdd}
\end{eqnarray}
Thus, with high probability, the number of times we visit state $A_m(t)=0$ is within $N(T)\pi_{m0} \pm \sqrt{N(T)}\log(V)$. This implies that we try to sample the transition $0\rightarrow1$ at least $N_{m0}=N(T)\pi_{m0} -\sqrt{N(T)}\log(V)$ times with high probability, and each time we succeed with probability $\epsilon_m$. 
Thus, using the concentration bound in \cite{chung_concentration} for i.i.d. Bernoulli variables, we have that the number of times  transition $0\rightarrow1$ takes place, denoted by $N_{m01}$, satisfies: 
\begin{eqnarray}
\hspace{-.2in}&& \prob{|N_{m01} - \epsilon_mN_{m0}| \geq 0.95\sqrt{N(T)}\log(V)  } \label{eq:prob-deviation-1} \\
\hspace{-.2in}&&\,\,   \leq 2e^{\frac{- (0.95)^2N(T)\log(V)^2}{ 2 (\epsilon_mN_{m0} -\sqrt{N(T)}\log(V) + 0.95\sqrt{N(T)}\log(V) /3)  }}\leq 2e^{-\log(V)^2/2}.  \nonumber
\end{eqnarray}
Combining (\ref{eq:pi-0-bdd}) and (\ref{eq:prob-deviation-1}), we conclude that: 
\begin{eqnarray}
\hspace{-.2in}&&\prob{|\hat{\epsilon}_m - \epsilon_m|\leq \frac{\log(V)}{\sqrt{N(T)}}}\label{eq:prob-deviation-2} \\
\hspace{-.2in}&&\qquad\qquad\geq 1- 2e^{-\log(V)^2/2} -  c_{m1}e^{-c_{m2}\log(V)^2} \nonumber \\
\hspace{-.2in}&&\qquad\qquad\geq 1- e^{-\log(V)^2/4}. \nonumber 
\end{eqnarray}
We can now repeat the above argument for $\hat{\delta}_m$ to obtain a similar result. Then,  (\ref{eq:error-bdd}) can be obtained by applying the union bound. 
\end{proof}

\begin{proof} (Lemma \ref{lemma:multiplier-bdd})
%
In  \lbisc,  we actually solve $\hat{g}_{\hat{\bv{\pi}}}(\gamma)=\sum_h\hat{\pi}_h\hat{g}_{h}(\gamma)$ for computing $\gamma^*_T$. Here $\hat{g}_{h}(\gamma)$ is due to the fact that we use the estimated values $\hat{\delta}_m$ and $\hat{\epsilon}_m$ in the dual function. 

Define $e(\pi_h)=\hat{\pi}_h - \pi_h $ and $e(a^{(i_h)}_m) = \hat{a}^{(i_h)}_m -a^{(i_h)}_m$. 
Then, we can rewrite $\hat{g}_{\hat{\bv{\pi}}}(\gamma)$ as:
\begin{eqnarray}
\hspace{-.2in}&& \hat{g}_{\hat{\bv{\pi}}}(\gamma) \triangleq  \sum_h [\pi_h + e(\pi_h)]\sum_m\sup_{\bv{\mu}^{(h)}_p}\bigg\{ V[a^{(i_h)}_m +e(a^{(i_h)}_m) ]  \label{eq:dual-error}\\
\hspace{-.2in}&&\qquad\qquad\qquad\qquad\qquad\qquad \times[\mu^{(h)}_{mp}r_{mp} + (1- \mu^{(h)}_{mp})r_{mc}] \nonumber \\
\hspace{-.2in}&&\qquad  - \gamma  [C_m(\mu^{(h)}_{mp}, z_h)      +(1-\mu^{(h)}_{mp}) [a^{(i_h)}_m+e(a^{(i_h)}_m)] \overline{C}_m - \rho]\bigg\}. \nonumber
\end{eqnarray}

Using Lemma \ref{lemma:mle-bdd}, we see that when $V$ is large, with probability $1-2Me^{-\log(V)^2/4}$, $e(a^{(i_h)}_m)\leq \frac{\log(V)}{\sqrt{N(T)}}$, which also implies that the estimated $\hat{\pi}_h$ is within $O(\frac{\log(V)}{\sqrt{N(T)}})$ of ${\pi}_h$. In this case, 
Assumption \ref{assumption:bdd-LM} implies that there exist a set of actions that achieve: 
\begin{eqnarray}
 \sum_h\hat{\pi}_h\sum_{j=1}^3\theta^{(h)}_j\sum_m [C_m(\mu^{(h)}_{mpj}, z_h)   +(1-\mu^{(h)}_{mpj}) \hat{a}^{(i_h)}_m\overline{C}_m]\leq\rho_0, \nonumber
\end{eqnarray}
for some $\rho_0=\rho-\eta>0$ for some $\eta>0$. 
Using the fact that $e(a^{(i_h)}_m)\leq \frac{\log(V)}{\sqrt{N(T)}}$, we see that the same set of actions ensures that: 
\begin{eqnarray}
 \hspace{-.3in}&& \sum_h\hat{\pi}_h\sum_{j=1}^3\theta^{(h)}_j\sum_m [C_m(\mu^{(h)}_{mpj}, z_h) \\
 \hspace{-.3in}&&\qquad\qquad\qquad   +(1-\mu^{(h)}_{mpj}) \hat{a}^{(i_h)}_m\overline{C}_m]\leq\rho-\eta/2. \nonumber
\end{eqnarray}
Using Lemma 1 in \cite{huang-learning-sig-14}, this implies $\gamma^*_T\leq \xi\triangleq \frac{2Vr_{p}}{\eta}$. 
Therefore, we have: 
\begin{eqnarray*}
 \hspace{-.2in} &&\hat{g}_{\hat{\bv{\pi}}}(\gamma^*_T)\geq   g_{\bv{\pi}}(\gamma^*_T) -\sum_h|e(\pi_h)| M(Vr_p + \frac{2Vr_{p}}{\eta}C_{\max})\\ 
 \hspace{-.2in}&&\qquad \quad- \sum_m (V\max_m| e(a^{(i_h)}_m) |r_p + \frac{2Vr_{p}}{\eta} \max_m| e(a^{(i_h)}_m) | C_{\max} ). 
\end{eqnarray*}
Similarly, we have: 
\begin{eqnarray*}
 \hspace{-.2in} &&\hat{g}_{\hat{\bv{\pi}}}(\gamma^*)\leq   g_{\bv{\pi}}(\gamma^*) + \sum_h|e(\pi_h)| M(Vr_p + \frac{2Vr_{p}}{\eta}C_{\max})\\ 
 \hspace{-.2in}&&\qquad \quad + \sum_m (V\max_m| e(a^{(i_h)}_m) |r_p + \frac{2Vr_{p}}{\eta} \max_m| e(a^{(i_h)}_m) | C_{\max} ). 
\end{eqnarray*}
Denoting the last two terms as $e_{tot}$, we see that $e_{tot}=\Theta(V)$. Using $\hat{g}_{\hat{\bv{\pi}}}(\gamma^*_T)\leq \hat{g}_{\hat{\bv{\pi}}}(\gamma^*)$, we get: 
\begin{eqnarray}
g_{\bv{\pi}}(\gamma^*_T) \leq g_{\bv{\pi}}(\gamma^*) + 2 e_{tot}. 
\end{eqnarray}
Using the polyhedral property (\ref{eq:polyhedral}), we conclude that: 
\begin{eqnarray}
|\gamma^*_T - \gamma^* |\leq 2e_{tot}/\beta. 
\end{eqnarray}
Finally using the fact that $\max(|e(\pi_h)|, |e(a^{(i_h)}_m|)\leq \frac{\log(V)}{\sqrt{N(T)}}$ proves the bound for $|\gamma^*_T - \gamma^* |$. The bound for $|\hat{\gamma}^* - \gamma^*|$ can be similarly proven. 
\end{proof}

\end{document}